\documentclass[reqno]{amsart}
\usepackage{amssymb, amsmath, amsthm, amscd, mathrsfs}
\usepackage{verbatim}
\usepackage{a4wide}
\usepackage{bbm}
\usepackage{enumerate}
\usepackage{color}
\allowdisplaybreaks
\usepackage{hyperref}
\theoremstyle{plain}
\newtheorem{theorem}{Theorem}[section]
\theoremstyle{remark}
\newtheorem{remark}[theorem]{Remark}

\theoremstyle{plain}
\newtheorem{corollary}[theorem]{Corollary}

\newtheorem{proposition}[theorem]{Proposition}

\numberwithin{equation}{section}
\def\N{{\mathbb N}}

\def\R{{\mathbb R}}

\newcommand{\ds}{\displaystyle}
\newcommand\numberthis{\addtocounter{equation}{1}\tag{\theequation}}

\newcommand{\supp}{\text{\rm supp\,}}

\def\typeout#1{\message{^^J}\message{#1}\message{^^J}}
\newif\ifSRCOK \SRCOKtrue
\newcount\PAGETOP
\newcount\LASTLINE
\global\PAGETOP=1
\global\LASTLINE=-1
\def\EJECT{\SRC\eject}
\def\WinEdt#1{\typeout{:#1}}% WinEdt LOG MODE and I_{\mathbb{R}^n}NPUT
\gdef\MainFile{\jobname.tex}% ".tex" needed for MiKTeX
\gdef\CurrentInput{\MainFile}
\newcount\INPSP
\global\INPSP=0
\def\SRC{\ifSRCOK%
  \ifnum\inputlineno>\LASTLINE%
    \ifnum\LASTLINE<0%
      \global\PAGETOP=\inputlineno%
    \fi%
    \global\LASTLINE=\inputlineno%
    \ifnum\INPSP=0%
      \ifnum\inputlineno>\PAGETOP%
        
      \fi%
    \else%
      
    \fi%
  \fi%
\fi}
\def\PUSH#1{%
\SRC%
\ifnum\INPSP=0 \global\let\INPSTACKA=\CurrentInput \else%
\ifnum\INPSP=1 \global\let\INPSTACKB=\CurrentInput \else%
\ifnum\INPSP=2 \global\let\INPSTACKC=\CurrentInput \else%
\ifnum\INPSP=3 \global\let\INPSTACKD=\CurrentInput \else%
\ifnum\INPSP=4 \global\let\INPSTACKE=\CurrentInput \else%
\ifnum\INPSP=5 \global\let\INPSTACKF=\CurrentInput \else%
               \global\let\INPSTACKX=\CurrentInput \fi\fi\fi\fi\fi\fi%
\gdef\CurrentInput{#1}%
\WinEdt{<+ \CurrentInput}%
\global\LASTLINE=0%
\ifSRCOK\fi%
\global\advance\INPSP by 1}
\def\POP{%
\ifnum\INPSP>0 \global\advance\INPSP by -1  \fi%
\ifnum\INPSP=0 \global\let\CurrentInput=\INPSTACKA \else%
\ifnum\INPSP=1 \global\let\CurrentInput=\INPSTACKB \else%
\ifnum\INPSP=2 \global\let\CurrentInput=\INPSTACKC \else%
\ifnum\INPSP=3 \global\let\CurrentInput=\INPSTACKD \else%
\ifnum\INPSP=4 \global\let\CurrentInput=\INPSTACKE \else%
\ifnum\INPSP=5 \global\let\CurrentInput=\INPSTACKF \else%
               \global\let\CurrentInput=\INPSTACKX \fi\fi\fi\fi\fi\fi%
\WinEdt{<-}%
\global\LASTLINE=\inputlineno%
\global\advance\LASTLINE by -1%
\SRC}
\def\INPUT#1{\relax}
\let\originalxxxeverypar\everypar
\newtoks\everypar
\originalxxxeverypar{\the\everypar\expandafter\SRC}
\everymath\expandafter{\the\everymath\expandafter\SRC}
\output\expandafter{\expandafter\SRCOKfalse\the\output}
\newif\ifSRCOK \SRCOKtrue
\newcount\PAGETOP
\newcount\LASTLINE
\global\PAGETOP=1
\global\LASTLINE=-1
\gdef\MainFile{\jobname.tex}% ".tex" needed for MiKTeX
\gdef\CurrentInput{\MainFile}
\newcount\INPSP
\global\INPSP=0
\def\EJECT{\SRC\eject}
\def\WinEdt#1{\typeout{:#1}}% WinEdt LOG MODE and I_{\mathbb{R}^n}NPUT
\def\SRC{\ifSRCOK%
  \ifnum\inputlineno>\LASTLINE%
    \ifnum\LASTLINE<0%
      \global\PAGETOP=\inputlineno%
    \fi%
    \global\LASTLINE=\inputlineno%
    \ifnum\INPSP=0%
      \ifnum\inputlineno>\PAGETOP%
      \fi%
    \else%
    \fi%
  \fi%
\fi}
\def\PUSH#1{%
\SRC%
\ifnum\INPSP=0 \global\let\INPSTACKA=\CurrentInput \else%
\ifnum\INPSP=1 \global\let\INPSTACKB=\CurrentInput \else%
\ifnum\INPSP=2 \global\let\INPSTACKC=\CurrentInput \else%
\ifnum\INPSP=3 \global\let\INPSTACKD=\CurrentInput \else%
\ifnum\INPSP=4 \global\let\INPSTACKE=\CurrentInput \else%
\ifnum\INPSP=5 \global\let\INPSTACKF=\CurrentInput \else%
               \global\let\INPSTACKX=\CurrentInput \fi\fi\fi\fi\fi\fi%
\gdef\CurrentInput{#1}%
\WinEdt{<+ \CurrentInput}%
\global\LASTLINE=0%
\ifSRCOK\fi%
\global\advance\INPSP by 1}
\def\POP{%
\ifnum\INPSP>0 \global\advance\INPSP by -1  \fi%
\ifnum\INPSP=0 \global\let\CurrentInput=\INPSTACKA \else%
\ifnum\INPSP=1 \global\let\CurrentInput=\INPSTACKB \else%
\ifnum\INPSP=2 \global\let\CurrentInput=\INPSTACKC \else%
\ifnum\INPSP=3 \global\let\CurrentInput=\INPSTACKD \else%
\ifnum\INPSP=4 \global\let\CurrentInput=\INPSTACKE \else%
\ifnum\INPSP=5 \global\let\CurrentInput=\INPSTACKF \else%
               \global\let\CurrentInput=\INPSTACKX \fi\fi\fi\fi\fi\fi%
\WinEdt{<-}%
\global\LASTLINE=\inputlineno%
\global\advance\LASTLINE by -1%
\SRC}
\def\INPUT#1{\relax}
\let\OldINCLUDE=\include
\def\include#1{%Always ".tex" file type!
\EJECT%
\PUSH{#1.tex}%
\OldINCLUDE{#1}%
\POP}
\let\originalxxxeverypar\everypar
\newtoks\everypar
\originalxxxeverypar{\the\everypar\expandafter\SRC}
\everymath\expandafter{\the\everymath\expandafter\SRC}
\let\zzzxxxbibliography=\bibliography
\def\bibliography#1{\PUSH{\jobname.bbl}\zzzxxxbibliography{#1}\POP}
\output\expandafter{\expandafter\SRCOKfalse\the\output}
\begin{document}
\author{E. M. Ait Benhassi}
\address{E. M. Ait Benhassi, CRMEF, Marrakesh}
\email{m.benhassi@uca.ma}	

\author{Mohamed Fadili}
\address{M. Fadili, Cadi Ayyad University, Faculty of Sciences Semlalia, 2390, Marrakesh, Morocco}
\email{m.fadili@ced.uca.ma}

\author{Lahcen Maniar}
\address{L. maniar, Cadi Ayyad University, Faculty of Sciences Semlalia, 2390, Marrakesh, Morocco}
\email{maniar@uca.ma}

\title[Algebraic condition]{On Algebraic condition for null controllability of some coupled degenerate systems}


%\date\today

\begin{abstract}
	In this paper we will generalize the Kalman rank condition for the null controllability   to   $n$-coupled linear degenerate parabolic systems with constant coefficients,  diagonalizable diffusion matrix, and $m$-controls. For that  we prove a global Carleman estimate of the solution of a scalar $2n$-order equation then we infer from it an observability inequality for the  corresponding adjoint system, and thus the null controllability.
\end{abstract}

\maketitle

%\tableofcontents
	
\section{Introduction and Main result }
In this work, we focus the following problem
 \begin{equation}\label{syst1}
 \left\lbrace \begin{array}{lll}
 \partial_t Y=(\mathbf{D}\mathcal{M}   +A) Y + B v \mathbbm{1}_{\omega}& in & Q, \\
 \mathbf{C} Y=0   & on & \Sigma, \\
 Y(0)=Y_0\,\,  & in & (0,1),
 \end{array}
 \right.
 \end{equation}
 where $Q:=(0,T)\times(0,1)$, $\Sigma:=(0,T)\times\{0,1\}$,  $\omega\subset (0,1)$ is a nonempty open  control region,  $\mathbbm{1}_{\omega}$ denotes the characteristic function of $\omega$, $T>0$, $\mathbf{D}$ is a $n\times n$ matrix, $B$ is a $n\times m$ matrix,   $v~=~(v_1, \cdots, v_m)^*$ is the control and  $Y=(y_1,\cdots,y_n)^*$ is the state. In the sequel we denote also $Q_{\omega}:=(0,T)\times\omega$. The operator $\mathcal{M}$ is defined by 
 $  \mathcal{M} y= {\left(  a y_x \right)}_x$ for $y\in D(\mathcal{M} )\subset L^2(0,1)$. For  $Y=(y_1,\cdots,y_n)^*$,  $\mathcal{M}Y$ denotes $(\mathcal{M}y_1,\cdots,\mathcal{M}y_n)^*$. 
The function  $a$ is a diffusion coefficient which degenerates at $0$ (i.e., $a(0)=0$)  and which can be either  weak  degenerate (WD), i.e.,
  \begin{equation}
  \text{(WD)}\,\,	
  \begin{cases}
  (i)\,\,a \in \mathcal{C}([0,1])\cup\mathcal{C}^1((0,1]),\, a>0 \text{ in }(0,1],\,a(0)=0,\\
  (ii)\,\,\exists K\in[0,1)\text{ such that }xa'(x)\leqslant Ka(x), \;\,\forall x\in[0,1],
  \end{cases}
  \end{equation}
 or  strong degenerate (SD), i.e.,
  \begin{equation}
  \text{(SD)}\,\,
  \begin{cases}
  (i)\,\,a \in \mathcal{C}^1([0,1]),\, a>0 \text{ in }(0,1],\,a(0)=0,\\
  (ii)\,\,\exists K\in[1,2)\text{ such that }xa'(x)\leqslant Ka(x)\,\forall x\in[0,1], \\
  (iii)\begin{cases}
  \ds\exists\theta\in(1,K] x\mapsto\frac{a(x)}{x^{\theta}}\text{ is nondecreasing near }0, \text{ if }K>1,\\
  \ds\exists\theta\in(0,1) x\mapsto\frac{a(x)}{x^{\theta}}\text{ is nondecreasing near }0, \text{ if }K=1.
  \end{cases}
  \end{cases}
  \end{equation}
   The boundary  condition  $\mathbf{C}Y=0$ is either  $Y(0)=Y(1)=0$ in the weak degenerate  case $(WD)$ or
 $Y(1)=(aY_x)(0)=0$ in the strongly degenerate case $(SD)$.
It is well known that null controllability of non degenerate ($a>0$) parabolic systems have been widely studied over the last 40 years and there have been a great number of results.  In the case of one equation  ($n=1$),   the result was obtained by  A. V. Fursikov and O. Y. Imanuvilov \cite{Fursikov} and  G. Lebeau and L. Robbiano \cite{Leb_Rob}.  In the case of coupled systems  $n\geq 2$, M. Gonzalez-Burgos, L. de Teresa \cite{ref3} provided a null controllability result for a cascade   parabolic  system. Recently,  F. Ammar-Khodja et al. \cite{A2,A3}  obtained several results characterizing the null controllability  of fully coupled systems with $m$-control forces by a generalized Kalman rank condition.  

For degenerate systems (e.g., $a(0)=0$),  null controllability of one equation was studied in \cite{BOU,CaMaVa} and the references therein. The case  of two coupled equations ($n=2$),  cascade  systems  are considered  in \cite{de-Ca}, and  in \cite{bahm, hjjaj}   the authors have studied the null controllability of degenerate  non-cascade parabolic systems.  In the case $n>2$,  in a recent work \cite{FadiliManiar}, we have extended the null controllability results obtained by Ammar-Khodja et al. \cite{A3} to a class of parabolic degenerate systems \eqref{syst1} in the two following cases :

\begin{enumerate}
	\item the coupling matrix $A$ is a cascade one and the diffusion matrix $\mathbf{D}=diag(d_1,\cdots,d_n)$ where $d_i>0$, $i=1,\cdots,n$,
	\item the coupling matrix $A$ is a full matrix (non cascade)  and the diffusion matrix $\mathbf{D}=d I_n, \; d>0$.
\end{enumerate}

In the present paper,  we study the case where  the coupling matrix $A$ is a full matrix and the diffusion matrix $\mathbf{D}$ is a diagonalizable $n\times n$ matrix with positive real eigenvalues, i.e.,
\begin{equation}\label{D:diagonalizable}
\mathbf{D}=P^{-1} \mathbf{J} P,\,\, P\in \mathcal{L}(\R^n),\,\, det(P)\neq 0,
\end{equation}
where  $\mathbf{J}=diag(d_1,\cdots,d_n)$, $d_i>0, 1\leqslant i\leqslant n$. The strategy used in this case is quite different from the one used in \cite{FadiliManiar}, and follows the one used in \cite{A4}.  To establish  an observability inequality to the adjoint system of  \eqref{syst1},  we prove a global Carleman estimate for a degenerate scalar equation \eqref{scalar:parabolic} of $2n$ order in space. This will lead to  several Carleman estimates, and thus to an observability inequality,  for our adjoint system.  Another difference of  \cite{FadiliManiar} is that  in  the Carleman estimates used for one degenerate equation, here  we need to  establish ones involving   the terms  $y_t$ and $(a(x)y_x)_x$ in addition to the state $y$ and its space derivative $y_x$.

 Let us introduce the following weighted spaces.
In  the (WD) case :
 $$\displaystyle H_{a}^1=\left\lbrace u \in L^2(0,1)/  u \text{   absolutely continuous in} [0,1], \sqrt{a}u_x \in L^2(0,1) \text{ and } u(1)=u(0)=0 \right\rbrace$$
 and
 $$\displaystyle H_{a}^2=\left\lbrace u \in H_{a}^1(0,1)/ au_x \in H^1(0,1)  \right\rbrace.$$
 In the (SD) case :
 $$\displaystyle H_{a}^1=\left\lbrace u \in L^2(0,1)/  u \text{   absolutely continuous in} (0,1], \sqrt{a}u_x \in L^2(0,1) \text{ and } u(1)=0 \right\rbrace$$
 and
 \begin{align*}
 \displaystyle H_{a}^2 &=\left\lbrace u \in H_{a}^1(0,1)/ au_x \in H^1(0,1)  \right\rbrace\\
 &=\left\lbrace u \in L^2(0,1)/  u \text{   absolutely continuous in} (0,1], au \in H_0^1(0,1),au_x \in H^1(0,1)\text{ and } (au_x)(0)=0  \right\rbrace.
 \end{align*}
  In both cases, the norms are defined as follow
 \begin{equation}\label{norms}
 	{\parallel u\parallel}_{H_{a}^1}^2 ={\parallel u\parallel}_{L^2(0,1)}^2+{\parallel \sqrt{a} u_x\parallel}_{L^2(0,1)}^2, \quad{\parallel u\parallel}_{H_{a}^2}^2 ={\parallel u\parallel}_{H_{a}^1}^2+{\parallel {(a u_x)}_x\parallel}_{L^2(0,1)}^2.
  \end{equation}
 Using the assumptions on the operator $\mathcal{M}$ and the condition \eqref{D:diagonalizable} on the diffusion matrix $\mathbf{D}$, for every $Y_0\in L^2(0,1)^n$ and $v\in L^2((0,T)\times(0,1))^m$, system \eqref{syst1}
  possesses a unique solution $\ds Y \in L^2(0,T;H_a^1(0,1)^n)\cap\mathcal{C}^0([0,T];L^2(0,1)^n)$. \\
Let us denote  $L:=\mathbf{D}\mathcal{M}+A$, with $D(L)=D(\mathcal{M})^n={H^2_a(0,1)}^n.$
 Then the  Kalman operator associated with $(L,B)$ is the matrix operator
 \begin{equation*}
 \begin{cases}
 \mathcal{K}:= [L|B] ~:~D(\mathcal{K})\subset L^2(0,1)^{nm}\longrightarrow L^2(0,1)^n, \\
 D(\mathcal{K}):=\Big\{ u \in L^2(0,1)^{nm}~:~\mathcal{K}u\in L^2(0,1)^n \Big\},
 \end{cases}
 \end{equation*}
 where
 \begin{equation*}
 [L|B]:= \big[L^{n-1}B|L^{n-2}B|\cdots|LB|B\big].
 \end{equation*}
 The adjoint system associated to the system \eqref{syst1} is the following
 \begin{equation}\label{syst2adjoint}
 \left\lbrace \begin{array}{lll}
 -\partial_t \varphi=\mathbf{D}^* \mathcal{M}\varphi  +A^*\varphi  & in & Q,\\
 \mathbf{C}\varphi=0  & on & \Sigma, \\
 \varphi(T)=\varphi_T\,\,  & in & (0,1).
 \end{array}
 \right.
 \end{equation}

 To study the null controllability  of the system \eqref{syst1}, we need to establish an observability inequality of the corresponding adjoint problem \eqref{syst2adjoint}. Indeed, we must prove the existence of a positive constant $C$ such that, for every $\varphi_0\in L^2(0,1)^n$,  the solution $\varphi\in\mathcal{C}^{0}([0,T],L^2(0,1)^n)$ of system \eqref{syst2adjoint} satisfies
 \begin{equation}\label{obser:inequality}
 \|\varphi(0,\cdot)\|_{L^2(0,1)^n}^2\leqslant C \int\!\!\!\!\!\int_{(0,T)\times\omega} |B^* \varphi(t,x)|^2 dx dt.
 \end{equation}
 The inequality \eqref{obser:inequality} will be deduced from a global Carleman estimate satisfied by the solution of the adjoint system \eqref{syst2adjoint} (Corollary \ref{coro4.3}\,). To prove this, we first show a Carleman estimate (Theorem \ref{thm4.2}\,) which bounds a weighted global integral of $\mathcal{K}^*\varphi$ by means of a weighted local integral of $B^*\varphi$. This last Carleman estimate is obtained by showing several intermediate Carleman estimates, and by assuming  the generalized Kalman  condition $Ker (\mathcal{K}^*)=\{0\}$, we will be able to obtain the desired Carleman estimate for system \eqref{syst2adjoint}. Thus, we conclude with the observability inequality  \eqref{obser:inequality} and the  null-controllability  of system \eqref{syst1}. At the end, we show that the generalized Kalman  condition $Ker (\mathcal{K}^*)=\{0\}$ is also necessary. Thus our  main result is the following.
 \begin{theorem}\label{maintheorem}
 	Let us assume that $\mathbf{D}$ satisfies \eqref{D:diagonalizable}. Then, system \eqref{syst1} is null controllable at any time $T>0$ if and only if the
 	Kalman operator $\mathcal{K}$ satisfies
 	\begin{equation}\label{K:null}
 	Ker (\mathcal{K}^*) = \{0\}.
 	\end{equation}	
 \end{theorem}

 The rest of the work is organized as follows: In section 2, we state some properties of the unbounded operator $\mathcal{K}$ and  give a useful characterization  of the Kalman condition $\ds Ker (\mathcal{K}^*) = \{0\}$ by using the spectrum of operator $\mathcal{M}$. Section 3 is devoted to show several intermediate Carleman estimates for  scalar parabolic  degenerate equations of order $2$ and  $2n$ in space. In Section 4, the proof of Theorem \ref{maintheorem}  is given in the end of Section 4. 

All along the article, we use generic constants for the estimates, whose values
 may change from line to line.

\section{ Spectrum of operator $\mathcal{M}$ and some algebraic tools}
This section will be devoted to prove two crucial properties of the Kalman operator $\mathcal{K}$ and to give an equivalent algebraic condition to the condition \eqref{K:null}. Let us focus on the spectrum of the unbounded operator $\mathcal{M}$ defined by $\forall u\in D(\mathcal{M})~:~ \mathcal{M}u=(a(x)u_x)_x$ where $D(\mathcal{M})=H_{a}^2(0,1)\subset L^2(0,1)$.\\
We recall the Hardy-poincar\'e inequality \cite[Proposition 2.1]{BOU}
\begin{proposition}
	For all $u$ in $H_{a}^2(0,1)$
	\begin{equation}\label{HarPoin}
		\int_{0}^{1}\frac{a(x)}{x^2} u^2(x)dx \leqslant C \int_{0}^{1} a(x) |u_x(x)|^2dx.
	\end{equation}
\end{proposition}

It is known that the operator $-\mathcal{M}$ is a definite positive operator. We will use the fact that $\ds H_{a}^2(0,1)$ is compactly embedded in $L^2(0,1)$, see \cite{cmp,Meyer}. 
 Thus, $-\mathcal{M}$ is a self-adjoint positive definite operator  with compact resolvent. 
Therefore, there exists a Hilbertian basis  $(\varPhi_n)_{n\in\N^*}$ of $L^2(0,1)$ and a sequence  $(\lambda_p)_{p\in\N^*}$ of real numbers with $\lambda_n>0$ and $\lambda_n\longrightarrow +\infty$, such that
\begin{equation}
	-\mathcal{M}\varPhi_n=\lambda_n\varPhi_n \qquad \forall n \in \N^*
\end{equation}
\begin{remark}
	In the case $a(x)=x^{\alpha}$ with $0<\alpha\leqslant 1$ as in \cite{Gueye} the eigenfunctions and eigenvalues of $\mathcal{M}$ can be explicitly given  using Bessel's functions.
\end{remark}
Now, we give some algebraic tools. It is  known that  $\ds \mathbb{D}:=\underset{p\geqslant 0}{\cap} D(\mathcal{M}^p)$ is dense in $D(\mathcal{M}^p)$ for every $p\geqslant 0$ and   $\mathbb{D}^{nm}\subset D(\mathcal{K})$. Thus,  $\overline{ D(\mathcal{K})}=L^2(0,1)^{nm}$ and  $\mathcal{K}^*$ is well defined from $D(\mathcal{K}^*)\subset L^2(0,1)^n$ into $L^2(0,1)^{nm}$.  The formal adjoint of $\mathcal{K}$, again  denoted by $\mathcal{K}^*$ is given by
 \begin{equation*}
 	 \mathcal{K}^*=\left[\begin{array}{l}
 	 B^* (L^*)^{n-1}\\
 	 \vdots\\
 	 B^* L^*\\
 	 B^*\\
 	 \end{array}\right],
 \end{equation*}
and it coincides with the adjoint operator of $\mathcal{K}$ on $\mathbb{D}^n$. Moreover, we note that when $a\in \mathcal{C}^{\infty}([0,1])$, from \cite[Proposition 3.8]{cmp},   $\mathbb{D}=\mathcal{C}^{\infty}([0,1])$. Thereafter,  we recall some properties of the Kalman operator $\mathcal{K}$ as it is given in \cite{A4}. For any $j, p \in \N^*$, we consider the projection operator
\begin{equation*}
	 P_p^j~:~\Psi=(\Psi_k)_{1\leqslant k\leqslant j}\in L^2(0,1)^j\rightarrow P_p^j(\Psi)=\left((\Psi_k,\varPhi_p)\right)_{1\leqslant k\leqslant j}\in\R^j ,
\end{equation*}
where $(\cdot,\cdot)$ stands for the scalar product in $L^2(0,1)$. All along this paper,  we denote by $|\cdot |$ the euclidian norm in $\R^j$. Thus, if $j\in\N^*$, 
we have the follwing characterization of $\mathbb{D}^j$
\begin{equation*}
	\mathbb{D}^j=\Big\{\Psi=\sum\limits_{p\geqslant 1}\Psi_p\varPhi_p~:\Psi_p\in\R^j\text{  and }\sum\limits_{p\geqslant 1}\lambda_p^{2m}|\Psi_p|^2<\infty,\,~\forall m\geqslant0 \Big\}.
\end{equation*}
For $p\in\N^*$,  $L_p:=-\lambda_p \mathbf{D}+A\,\in\mathcal{L}(\R^n)$ and
$$\mathcal{K}_p=[L_p|B]=\big[L_p^{n-1}B|\cdots|L_pB|B\big]\,\in\mathcal{L}(\R^{nm},\R^n) .$$
We have the following equalities
\begin{equation*}
	\begin{cases}
	 L(b\varPhi_p)=(L_pb)\varPhi_p,\quad  b\in\R^n,\, p\geqslant 1 ,\\
	  \mathcal{K}(b\varPhi_p)=(\mathcal{K}_pb)\varPhi_p, \quad  b\in\R^{n m},\,p\geqslant 1. \\
	\end{cases}
\end{equation*}
Since   $L$ and $\mathcal{K}$ are closed unbounded operators, one has
\begin{equation*}
	\begin{cases}
	Ly=\sum\limits_{p\geqslant 1} L_p P_p^n(y)\varPhi_p,\,\, \forall y \in D({L}),\\
	\mathcal{K}u=\sum\limits_{p\geqslant 1}\mathcal{K}_p P_p^{n m}(u)\varPhi_p,\,\, \forall u \in D(\mathcal{K}),
	\end{cases}
\end{equation*}
and then
\begin{equation*}
	D(\mathcal{K})=\Big\{u\in L^2(0,1)^{n m}~:~\sum\limits_{p\geqslant 1}{|\mathcal{K}_p P_p^{n m}(u)|}^2<\infty  \Big\}.
\end{equation*}
In a similar way, we obtain
\begin{equation*}
	\begin{cases}
	 	\mathcal{K}^*=\sum\limits_{p\geqslant 1}\mathcal{K}_p^* P_p^n(\cdot)\varPhi_p\,\,\text{  with}\,\, \mathcal{K}_p^*={[L_p|B]}^*,\\
	 	D(\mathcal{K}^*)=\Big\{\varphi\in L^2(0,1)^{n}~:~\sum\limits_{p\geqslant 1}{|\mathcal{K}_p^* P_p^{n}(\varphi)|}^2<\infty  \Big\}.
	\end{cases}
\end{equation*}
We define also the operator $\mathcal{K}\mathcal{K}^*~:~D(\mathcal{K}\mathcal{K}^*)\subset L^2(0,1)^n\longrightarrow L^2(0,1)^n,$ with domain
$$D(\mathcal{K}\mathcal{K}^*)=\Big\{\varphi\in L^2(0,1)^n\,\mathcal{K}^*\varphi\in D(\mathcal{K}),\,\mathcal{K}\mathcal{K}^*\varphi\in L^2(0,1)^n  \Big\}. $$

The operator $\mathcal{K}\mathcal{K}^*$ is  closed,   and a simple computation provides
\begin{equation*}
\begin{cases}
\mathcal{K}\mathcal{K}^*\varphi=\sum\limits_{p\geqslant 1}\mathcal{K}_p\mathcal{K}_p^* P_p^n(\varphi)\varPhi_p,\\
D(\mathcal{K}\mathcal{K}^*)=\Big\{\varphi\in L^2(0,1)^{n}~:~\sum\limits_{p\geqslant 1}{|\mathcal{K}_p^* P_p^{n}(\varphi)|}^2\, ,\,\sum\limits_{p\geqslant 1}{|\mathcal{K}_p\mathcal{K}_p^* P_p^{n}(\varphi)|}^2<\infty  \Big\}.
\end{cases}
\end{equation*}
As in \cite{A4}, we obtain  the following result.
\begin{proposition}\label{prop22}
	The following conditions are equivalent
	\begin{enumerate}
		\item $Ker (\mathcal{K}^*)=\{0\}.$
		\item $Ker (\mathcal{K}\mathcal{K}^*)=\{0\}.$
		\item $det (\mathcal{K}_p\mathcal{K}_p^*)\neq 0\,\,\, \text{ for every } p\geqslant 1.$
	\end{enumerate}
\end{proposition}
\begin{proof}
	To show  $(1)\Rightarrow (2)$, assume that $Ker (\mathcal{K}\mathcal{K}^*)\neq\{0\}$. Then,  there exists a non-zero element $v\in L^2(0,1)^n$ such that $\mathcal{K}\mathcal{K}^* v=0$. Thus, $||\mathcal{K}^* v||^2=0 $.  Therefore, $ \mathcal{K}^* v=0$ and this contradicts $(1)$. The implication
	 $(2)\Rightarrow (1)$ follows  from $Ker (\mathcal{K}^*)\subset Ker (\mathcal{K}\mathcal{K}^*)$.  For
	 $(2)\Rightarrow (3)$, assume  $det (\mathcal{K}_{p_0}\mathcal{K}_{p_0}^*)= 0$ for some ${p_0}\geqslant 1$. Then $0$ is an eigenvalue of $\mathcal{K}_{p_0}\mathcal{K}_{p_0}^*$. Thus, there exists a non-zero vector $v_{p_0}\in \R^n$ such that $\mathcal{K}_{p_0}\mathcal{K}_{p_0}^* v_{p_0}=0$. Therefore,  $\varphi= v_{p_0}\varPhi_{p_0}$ is a non zero element in $D(\mathcal{K}\mathcal{K}^*) $ such that
	 \begin{align*}
	  \mathcal{K}\mathcal{K}^*(\varphi)&= \sum\limits_{p\geqslant 1}\mathcal{K}_p\mathcal{K}_p^* P_p^n(\varphi)\varPhi_p= \mathcal{K}_{p_0}\mathcal{K}_{p_0}^* v_{p_0}\varPhi_{p_0}
	  =0.
	 \end{align*}
	 This would contradict the result $Ker (\mathcal{K}\mathcal{K}^*)=\{0\}$. 	 Finally, to 
$(3)\Rightarrow (2)$, let $y\in Ker (\mathcal{K}\mathcal{K}^*)$. Hence 
	 	$\sum\limits_{p\geqslant 1}\mathcal{K}_p\mathcal{K}_p^* P_p^n(y)\varPhi_p =0$, and then  $ \mathcal{K}_p\mathcal{K}_p^* P_p^n(y) =0$ for all $p\geqslant 1$. 
	 	Therefore, $P_p^n(y) =0$,  since  $det (\mathcal{K}_p\mathcal{K}_p^*)\neq 0 $ for all $p\geqslant 1$. Thus,   
	 	$y=\sum\limits_{p\geqslant 1} P_p^n(y)\varPhi_p =0.$
\end{proof}
The previous proposition is of great interest, since it allows us to check the following Theorem whose the proof, in our degenerate case,  is similar to  \cite[Theorem 2.1]{A4}.
\begin{theorem}\label{theorem2.5}
We have  the following properties
\begin{enumerate}
	\item there exists a constant $ C>0$ such that for all $u\in {D(\mathcal{M}^{n-1})}^{nm}, \,\, \mathcal{K}u\in L^2(0,1)^n$ and
		$${ \|\mathcal{K}u \| }^2_{L^2(0,1)^n}\leqslant C { \|\mathcal{M}^{n-1} u \| }^2_{L^2(0,1)^{n m}},$$
\item there exists a constant $ C>0$ such that for all $u\in {D(\mathcal{M}^{n-1})}^{n}, \,\, \mathcal{K}^*u\in L^2(0,1)^{n m}$ and
$${ \|\mathcal{K}^*u \| }^2_{L^2(0,1)^{n m}}\leqslant C { \|\mathcal{M}^{n-1} u \| }^2_{L^2(0,1)^{n}},$$
\item assume $Ker (\mathcal{K}^*)=\{0\}$, and  let $k\geqslant (2n-1)(n-1)$. Then, for every $\varphi \in L^2(0,1)^n$ satisfying $\mathcal{K}^*\varphi\in {D(\mathcal{M}^k)}^{n m}$, one has $\varphi \in  {D(\mathcal{M}^{k-(2n-1)(n-1)})}^{n}$ and
$${\|\mathcal{M}^{k-(2n-1)(n-1)}\varphi \|}^2_{L^2(0,1)^n}\leqslant C  {\|\mathcal{M}^k \mathcal{K}^*\varphi \|}^2_{L^2(0,1)^{n m}}.$$
\end{enumerate}	
\end{theorem}
By adapting the proof of  \cite[Theorem 2.1]{A4} to our case and using the fact that the polynomial $F(\lambda): = $  is either  identically $0$ or far from $0$ for any $\lambda$ sufficiently large, one can deduce the following corollary:
\begin{corollary}
	Either there exists $p_0\in\N^*$ such that $rank~\mathcal{K}_p=n$ for every $p>p_0$ or $rank~\mathcal{K}_p<n$ for every $p\in\N^*$
\end{corollary}

\section{Carleman estimates}
In this section we give a new global Carleman estimate for the adjoint problem \eqref{syst2adjoint}.
By the same way as in \cite[Proposition 3.3]{A4}, we can show the following result.
\begin{proposition}\label{propo:4.1}
	If $\varphi=(\varphi_1,\cdots,\varphi_n)^*$ is the solution of problem \eqref{syst2adjoint}  corresponding to initial data $\varphi_0\in X:=\mathbb{D}^n$, then $\varphi \in \mathcal{C}^k([0,T]; D(\mathcal{M}^p)^n) $ for every $k,\, p\geqslant 0$, and
	\begin{equation}
	det (I_d\partial_t+(\mathbf{D}\mathcal{M}+A)^*)\varphi_i=0\qquad \text{in } Q,\,1\leqslant i\leqslant n.
	\end{equation}
\end{proposition}
In order to state our fundamental result, we need to show first some Carleman estimates in the case of a single parabolic degenerate equation.
\subsection{Carleman estimate for one equation}\noindent\newline
In this subsection  we shall establish  a new  Carleman estimate for the solution
of the following  parabolic equation 
\begin{equation}\label{problem}
\begin{cases}
u_t-(a(x)u_x)_x+\tilde{c}u=f,\quad (t,x)\in Q, \\
	u(t,1)=0 \text{ and }\begin{cases}
	\ds u(t,0)=0,\,\,\text{ in case (WD)}  \\
	\ds  {(a(x) u_{ x})}(t,0)=0,\,\,\text{ in case (SD) }
	\end{cases}\\
u(0,x)=u_0(x),   \quad x\in (0,1).
\end{cases}
\end{equation}
Let us consider the following time and space weight functions
 \begin{equation}\label{fonctionspoids}
 \begin{cases}
 \ds \theta(t)=\frac{1}{t^4(T-t)^4}\,\,, \psi(x)=\lambda\left(\int_0^x\frac{y}{a(y)}dy-c\right)\,\,\text{ and }\, \varphi(t,x)=\theta(t)\psi(x),\\
 \Phi(t,x)=\theta(t)\Psi(x)\,\,  \,\text{ and  }\,\, \Psi(x)=e^{\rho\sigma(x)} -e^{2\rho{\parallel\sigma\parallel}_{\infty}},
 \end{cases}
 \end{equation}

  where the parameters $c$, $\rho$ and $\lambda$ are  chosen as in \cite{hjjaj}

  	\begin{equation*}
  		\begin{cases}
  			\ds c>5 \,\,,\quad \,\, \rho>\frac{4\ln2}{{\parallel\sigma\parallel}_{\infty}}.\\
  			\ds\frac{e^{2\rho{\parallel\sigma\parallel}_{\infty}}}{c-1}<\lambda<\frac{4}{3c}{\left( e^{2\rho{\parallel\sigma\parallel}_{\infty}}-e^{\rho{\parallel\sigma\parallel}_{\infty}}\right)}.
  		\end{cases}
  	\end{equation*}
The following Carleman estimate will be crucial for the aim of this subsection. Note that the Carleman estimate needed in this work is  different from the one showed in \cite{BOU} and used in \cite{FadiliManiar},  since it  involves in addition to $u$ and $u_x$ the terms $u_t$ and $\mathcal{M}u$. 
\begin{theorem} \label{thm3.1}
	let $T>0$.
	Then there exist two positive constants $C$ and $s_0$ such that, for all $\ds u_0\in H_a^1$,
	the solution $u$ of equation \eqref{problem} satisfies,
	\begin{equation} \label{equ:3.28}
	\begin{aligned}
	&\int\!\!\!\!\!\int_{Q} \Big(\frac{1}{s\theta}(u_t^2+(\mathcal{M}u)^2)+s\theta a(x)u_x^2+s^3 \theta^3
	\frac{x^2}{a(x)}u^2 \Big)e^{2s\varphi}\,dx\,dt  \\
	& \leqslant C\Big( \int\!\!\!\!\!\int_{Q} f^2 e^{2s\varphi}\,dx\,dt
	+ s a(1)\int_{0}^T \theta u_x^2(t,1)e^{2s\varphi} dt\Big)
	\end{aligned}
	\end{equation}
	 for all $s\geq s_0$.
\end{theorem}

\begin{proof}
	Let $u$ be the solution of equation \eqref{problem}. For $s>0$,  the
	function $w=e^{s\varphi}u$ satisfies
	\[
	\underbrace{-(aw_x)_x-s\varphi_tw-s^2a\varphi_x^2w}_{L_s^+w}
	+\underbrace{w_t+2sa\varphi_xw_x+s(a\varphi_x)_xw}_{L_s^-w}
	=\underbrace{fe^{s\varphi}-\tilde{c}w}_{f_s}.
	\]
	Moreover, from the  Lemma 3.4, 3.5 and 3.6 in \cite{BOU}, we can deduce the following estimate to $w$  in $(0,T)\times (0,1)$
\begin{equation}\label{equ:3.291}
		\begin{aligned}
		& \|L_s^+w\|^2+\|L_s^-w\|^2+\iint_{Q}
		\Big(s^3 \theta^3\frac{x^2}{a(x)}w^2+s\theta a(x)w_x^2 \Big)\,dx\,dt  \\
		& \leqslant C\Big( \iint_{Q} (fe^{s\varphi}-\tilde{c}w)^2\,dx\,dt
		+ sa(1)\int_{0}^T \theta(t) w_x^2(t,1) dt\Big)\\
		& \leqslant C\Big( \iint_{Q} f^2e^{2s\varphi}+\iint_{Q}\tilde{c}^2w^2 \,dx\,dt
		+ sa(1)\int_{0}^T \theta(t) w_x^2(t,1) dt\Big).\\
		\end{aligned}
\end{equation}
	Using the same technique as in \cite{hjjaj} and \cite{BOU} the term $\ds\iint_{Q}\tilde{c}^2w^2 \,dx\,dt$ can be absorbed by the last two terms in the left side of inequality \eqref{equ:3.291}. Thus
		\begin{equation}\label{equ:3.29}
		\begin{aligned}
		& \|L_s^+w\|^2+\|L_s^-w\|^2+\iint_{Q}
		\Big(s^3 \theta^3\frac{x^2}{a(x)}w^2+s\theta a(x)w_x^2 \Big)\,dx\,dt  \\
		& \leqslant C\Big( \iint_{Q} f^2 e^{2s\varphi}\,dx\,dt
		+ sa(1)\int_{0}^T \theta(t) w_x^2(t,1) dt\Big).
		\end{aligned}
		\end{equation}
	Using the previous estimate, we will bound the integral
	$\ds \int\!\!\!\!\!\int_{Q}\frac{1}{s\theta}u_t^2 e^{2s\varphi}\,dx\,dt$.
	In fact, we have
	\begin{align*}
	\frac{1}{\sqrt{s\theta}}L_s^-w &=\frac{1}{\sqrt{s\theta}}(w_t+2sa\varphi_xw_x+s(a\varphi_x)_xw) \\
	& =\frac{1}{\sqrt{s\theta}}w_t+2 \sqrt{s\theta}x w_x+\sqrt{s\theta}w.
	\end{align*}
	Therefore,
	\begin{align}\label{equ:3.30}
	\iint_{Q}\frac{1}{s\theta}w_t^2 \,dx\,dt&\leqslant C\left(\|L_s^-w\|^2+\iint_{Q} s\theta\frac{x^2}{a}aw_x^2\,dx\,dt+\iint_{Q} s\theta w^2\,dx\,dt \right).
	\end{align}
	Since the function $\ds x \longmapsto \frac{x^2}{a}$ is nondecreasing, then one has
	\begin{equation}\label{equ:3.31}
	\iint_{Q} s\theta\frac{x^2}{a}aw_x^2\,dx\,dt\leqslant \frac{1}{a(1)}\iint_{Q} s\theta a w_x^2\,dx\,dt.
	\end{equation}
	Thanks to the Hardy-Poincar\'e inequality \eqref{HarPoin}, we can estimate $\ds \iint_{Q} s\theta w^2\,dx\,dt$ as follow
	\begin{align*}
	\iint_{Q} s\theta w^2\,dx\,dt
	& = s \iint_{Q} \Big( \theta\frac{a^{1/3}}{x^{2/3}}w^2
	\Big)^{3/4}\Big( \theta\frac{x^2}{a}w^2\Big)^{1/4}\,dx\,dt \\
	& \leqslant s \frac32 \iint_{Q} \theta\frac{a^{1/3}}{x^{2/3}}w^2\,dx\,dt
	+ \frac{s}{2} \iint_{Q}  \theta\frac{x^2}{a}w^2 \,dx\,dt.
	\end{align*}
	Let $p(x)= x^{\frac{4}{3}} a^{\frac{1}{3}}$, then one has, $\ds p(x)=a\left(\frac{x^2}{a}\right)^{\frac{2}{3}}\leqslant C a(x)$ since the function $\ds x\longmapsto \frac{x^2}{a}$ is non-decreasing on $(0,1)$.
	\begin{equation}\label{equ:C_{HP}}
	\int_0^1\frac{a^{\frac{1}{3}}}{x^{\frac{2}{3}}} w^2 dx =\int_0^1 \frac{p(x)}{x^2}w^2 dx
	\leqslant C \int_0^1 \frac{a(x)}{x^2}w^2 dx
	\leqslant C \int_0^1 a(x)w_x^2 dx.
	\end{equation}
	
	Thus
	\begin{align}\label{equ:3.32}
	\iint_{Q} s\theta w^2\,dx\,dt
	&\leqslant C \iint_{Q} \Big( s\theta aw_x^2 +s^3\theta^3\frac{x^2}{a}w^2
	\Big)dx\,dt.
	\end{align}
		From \eqref{equ:3.30}-\eqref{equ:3.32}, we get
			\begin{align}\label{equcor:3.39}
			\int\!\!\!\!\!\int_{Q}\frac{1}{s\theta}w_t^2 \,dx\,dt&\leqslant C\left(\|L_s^-w\|^2+  \iint_{Q} s \theta a w_x^2\,dx\,dt+\int\!\!\!\!\!\int_{Q} s^3\theta^3\frac{x^2}{a} w^2 \,dx\,dt \right).
			\end{align}

	In a similar way, to bound the integral $\ds \int\!\!\!\!\!\int_{Q}\frac{1}{s\theta}(\mathcal{M}w)^2 \,dx\,dt$, we have
	\begin{align*}
	\frac{1}{\sqrt{s\theta}}L_s^+w &=\frac{1}{\sqrt{s\theta}}(-(aw_x)_x-s\varphi_tw-s^2a\varphi_x^2w ) \\
	& =\frac{-\mathcal{M}w}{\sqrt{s\theta}}-\sqrt{s}\frac{\dot{\theta}\psi}{\sqrt{\theta}}w-  s^{\frac32}\theta^{\frac32}\frac{x^2}{a}w.
	\end{align*}
	As $\ds\Big| \dot{\theta}{\theta}^{-1}(t)\Big| =\Big|\frac{-4(T-2t)}{t(T-t)}\Big|\leqslant 4T {\theta}^{1/4}$ and  for all $t\in [0,T],\, |\theta(t)|\geqslant (\frac{2}{T})^8 $, then we deduce  $\ds \Big|\frac{\dot{\theta}}{\sqrt{\theta}}\Big|\leqslant~C~\theta^{\frac32}$. Since the function $\psi$ is bounded on $(0,1)$ then
	$\ds \int\!\!\!\!\!\int_{Q}(\sqrt{s}\frac{\dot{\theta}\psi}{\sqrt{\theta}}w)^2dx\,dt\leqslant C\int\!\!\!\!\!\int_{Q} s\theta^{\frac32}w^2 dx\,dt$.
	By using inequality \eqref{equ:C_{HP}}, we infer
	\begin{align*}
	\iint_{Q} s\theta^{\frac32} w^2\,dx\,dt
	& = s \iint_{Q} \Big( \theta\frac{a^{1/3}}{x^{2/3}}w^2
	\Big)^{3/4}\Big( \theta^3\frac{x^2}{a}w^2\Big)^{1/4}\,dx\,dt \\
	& \leqslant s \frac32 \iint_{Q} \theta\frac{a^{1/3}}{x^{2/3}}w^2\,dx\,dt
	+ \frac{s}{2} \iint_{Q}  \theta^3\frac{x^2}{a}w^2 \,dx\,dt\\
	& \leqslant  C \frac32 \iint_{Q} s \theta a w_x^2\,dx\,dt
	+ \frac{s}{2} \iint_{Q}  \theta^3\frac{x^2}{a}w^2 \,dx\,dt.
	\end{align*}  	
	Therefore, for $s$ large enough
		$\ds \int\!\!\!\!\!\int_{Q}(\sqrt{s}\frac{\dot{\theta}\psi}{\sqrt{\theta}}w)^2dx\,dt\leqslant  C \left(  \iint_{Q} s \theta a w_x^2\,dx\,dt
		+  \iint_{Q}  s^3 \theta^3\frac{x^2}{a}w^2 \,dx\,dt \right)$.
	Thus
	\begin{align}\label{equ:3.39}
	\int\!\!\!\!\!\int_{Q}\frac{1}{s\theta}(\mathcal{M}w)^2 \,dx\,dt&\leqslant C\left(\|L_s^+w\|^2+  \iint_{Q} s \theta a w_x^2\,dx\,dt+\int\!\!\!\!\!\int_{Q} s^3\theta^3\frac{x^2}{a} w^2 \,dx\,dt \right).
	\end{align}
	From inequalities \eqref{equ:3.29}, \eqref{equcor:3.39} and \eqref{equ:3.39},  one obtains
	\begin{equation}\label{equ:3.33}
	\begin{aligned}
	& \int\!\!\!\!\!\int_{Q}
	\Big(\frac{1}{s\theta}w_t^2+\frac{1}{s\theta}(\mathcal{M}w)^2+s^3 \theta^3\frac{x^2}{a(x)}w^2+s\theta a(x)w_x^2 \Big)\,dx\,dt  \\
	& \leqslant C\Big( \int\!\!\!\!\!\int_{Q} f^2 e^{2s\varphi}\,dx\,dt
	+ s a(1)\int_{0}^T \theta(t) w_x^2(t,1) dt\Big).
	\end{aligned}
	\end{equation}
	Consequently, we obtain the estimate \eqref{equ:3.28} which completes the proof.
\end{proof}
From the boundary Carleman estimate \eqref{equ:3.28}, we deduce  Carleman estimates for equation \eqref{problem} on the subregion $\omega'$.  Set $\omega'':=(x_1^{''},x_2^{''}) \subset\subset\omega'$ and
$\xi\in \mathcal{C}^\infty([0,1])$ such that $0\leq \xi(x)\leq 1$
for $x\in(0,1)$, $\xi(x)=1$ for $x\in (0,x_1^{''})$ and $\xi(x)=0$
for $x\in (x_2^{''}, 1)$. 

 State first the following intermediate Carleman estimate.

\begin{proposition}\label{estima0a}
		Let $T>0$, there
		exist two positive constants $C$ and $s_0$ such that, for every $u_0\in L^2(0,1)$,
		the solution $u$ of equation \eqref{problem}	satisfies
		\begin{align}\label{estimleft}
		& \int\!\!\!\!\!\int_{Q} \Big( \frac{1}{s\theta}\xi^2 u_t^2+ \frac{1}{s\theta}\xi^2(\mathcal{M}u)^2+ s^3 \theta^3
		\frac{x^2}{a}\xi^2u^2 + s\theta a\xi^2 u_x^2
		\Big)e^{2s\varphi}dt dx
		\notag \\
		&\leqslant C \left( \int\!\!\!\!\!\int_{Q} \xi^2 f^2e^{2s\varphi} dtdx  +
		\int\!\!\!\!\!\int_{Q_{\omega'}} s^2\Theta^2
		u^2 e^{2s\varphi} dtdx \right)
		\end{align}
		for all $s\geq s_0$.
	\end{proposition}
	
\begin{proof}
	First, let $u_0\in H_a^1$. The function  $z:=\xi u$   satisfies the following equation
	\begin{equation}\label{equ:3.35}
	\begin{cases}
	z_t-(a(x) z_x)_x+\tilde{c}z= \xi f -\xi_x a(x) u_x-(a(x)\xi_xu)_x, & (t,x)\in Q,\\
	z(t,1)=0 \text{ and }\begin{cases}
	\ds z(t,0)=0,\,\,\text{ in case (WD)}  \\
	\ds  {(a(x) z_{ x})}(t,0)=0,\,\,\text{ in case (SD) }
	\end{cases}  &\text{ on }  (0,T),\\
	z(0,x)=\xi(x)u_0, & x\in (0,1).
	\end{cases}
	\end{equation}
	The Carleman estimate \eqref{equ:3.28} applied  to equation \eqref{equ:3.35} yields to
	\begin{equation} \label{equ:3.36}
	\begin{aligned}
	&\int\!\!\!\!\!\int_{Q} \Big(s\theta a(x)z_x^2+s^3 \theta^3
	\frac{x^2}{a(x)}z^2 +\frac{1}{s\theta}z_t^2+\frac{1}{s\theta}(\mathcal{M}z)^2 \Big)e^{2s\varphi}\,dx\,dt  \\
	& \qquad\qquad\leqslant C \int\!\!\!\!\!\int_{Q} \left( \xi^2 f^2 +(\xi_x a(x) u_x+(a(x)\xi_xu)_x)^2 \right)e^{2s\varphi}\,dx\,dt.
	\end{aligned}
	\end{equation}
From the definition of $\xi$ and the Cacciopoli inequality \cite[Lemma 6.1]{hjjaj}, we obtain
 \begin{align*}
 	 	\int\!\!\!\!\!\int_{Q} (\xi_x a(x) u_x+(a(x)\xi_xu)_x)^2 e^{2s\varphi}\,dx\,dt
 	 	&\leqslant C \int\!\!\!\!\!\int_{Q_{\omega''}}(u^2+u_x^2)e^{2s\varphi}\,dxdt\\
 	 	&\leqslant C \int\!\!\!\!\!\int_{Q_{\omega'}}(s^2\theta^2u^2+f^2)e^{2s\varphi}\,dxdt.\numberthis\label{equ:3.361}
 \end{align*}
Moreover, since  $\ds \xi u_x=z_x-\xi_xu$ and $\ds \xi\mathcal{M}u=\mathcal{M}z -(a\xi_x u)_x-\xi_x a u_x $,
 then we get 			
\begin{equation}\label{equ:3.362}
	\int\!\!\!\!\!\int_{Q}s\theta a\xi^2 u_x^2 e^{2s\varphi} dx dt \leqslant 2 \int\!\!\!\!\!\int_{Q}s\theta a z_x^2 e^{2s\varphi} dx dt+2 \int\!\!\!\!\!\int_{Q_{\omega'}}s^2\theta^2 u^2 e^{2s\varphi} dx dt
\end{equation}		
and
\begin{equation}\label{equ:3.363}
\int\!\!\!\!\!\int_{Q}\frac{1}{s\theta}\xi^2 \mathcal{M}u^2 e^{2s\varphi} dx dt \leqslant 2 \int\!\!\!\!\!\int_{Q}\frac{1}{s\theta} \mathcal{M} z^2 e^{2s\varphi} dx dt+C \int\!\!\!\!\!\int_{Q_{\omega'}}(s^2\theta^2u^2+f^2)e^{2s\varphi}\,dxdt.
\end{equation}	
Thus, from \eqref{equ:3.36}-\eqref{equ:3.363} and the definition of $\xi$ we deduce the desired estimate for $u_0\in H^1_a$.  Finally, by density, we conclude for  $u_0\in L^2$.			
 \end{proof}
Using the previous Carleman estimate \eqref{estimleft}, by the same argument of \cite[Proposition 2.4]{FadiliManiar}, we obtain the following general version.
\begin{proposition}\label{propo3.6}
		Let $T>0$ and  $\tau\in \mathbb{R}$. Then there exists two positive constants $C$ and $s_0$ such that, for all $\ds u_0\in L^2(0,1)$, the solution $u$ of equation \eqref{problem} satisfies
		\small{\begin{multline}\label{eqproposition23}
			\int\!\!\!\!\!\int_{Q} \left(s^{\tau-1}\theta^{\tau-1}\xi^2 u_t^2+s^{\tau-1}\theta^{\tau-1}\xi^2 (\mathcal{M}u)^2 +s^{1+\tau}\theta^{1+\tau} a \xi^2 u_x^2+s^{3+\tau} {\theta}^{3+\tau}\frac{x^2}{a}\xi^2  u^2\right)e^{2s\varphi(t,x)}dxdt \\
			\leqslant C\left(\int\!\!\!\!\!\int_{Q} \xi^2 s^\tau\theta^{\tau}f^2(t,x)e^{2s\varphi(t,x)}dx dt+\int\!\!\!\!\!\int_{Q_{\omega'}} s^{2+\tau}\theta^{2+\tau} u^2 e^{2s\varphi(t,x)}dxdt  \right)
			\end{multline}}
		for all $s\geqslant s_0$.
	\end{proposition}
Proposition \ref{estima0a} gave a Carleman estimate in
$(0,x_1^{'})$. For the interval $(x_1^{'},1)$, similarly as in \cite{A. HAJJAJ},   \cite[Lemma 1.2]{Fursikov} remains true when we replace $\theta=\frac{1}{t(T-t)}$ by $\theta=\frac{1}{t^4(T-t)^4}$. Thus,  we have the following lemma.  
Now, with the non generate Carleman estimate of \cite[Lemma 1.2]{Fursikov}, we are able to give a Carleman estimate to equation \eqref{problem} on the interval $(x_1',1)$.

\begin{proposition}\label{estimaa1}
	There exist two positive constants $C$ and $s_0$ such that for every $u_0\in L^2(0,1)$,  the solution $u$ of equation \eqref{problem} satisfies
	 \begin{align*}\label{estimright}
		& \int\!\!\!\!\!\int_{Q} \Big( \frac{1}{s\theta}\zeta^2 u_t+\frac{1}{s\theta}\zeta^2(\mathcal{M}u)^2+ s^3 \theta^3
		\frac{x^2}{a}\zeta^2 u^2 + s\theta a \zeta^2 u_x^2
		\Big)e^{2s\Phi}dt dx
		\notag \\
		&\qquad\qquad\leqslant C \Big( \int\!\!\!\!\!\int_{Q} \zeta^2
		f^2e^{2s\Phi} dtdx + \int\!\!\!\!\!\int_{Q_{\omega'}} \!\! s^3\theta^3 u^2 e^{2s\Phi} dtdx \Big)%\numberthis\label{prop381}
		\end{align*}
	for all $s\geq s_0$, where $\zeta:=1-\xi$ and $\varPhi=\theta\varPsi, \; \,
	\varPsi(x)=\big(e^{\rho\sigma(x)}-e^{2\rho{\|\sigma\|}_{\infty}}\big),$
	with  $\sigma$  a $\mathcal{C}^2([0,1])$ function such that $\sigma(x)>0$ in $(0,1)$, $\sigma(0)=\sigma(1)=0$ and $\sigma_x(x)\neq 0$ in $[0,1]\setminus \omega_0$, $\omega_0$ is an open subset of $\omega$.
\end{proposition}
\begin{proof}
	Not only the function $Z:=\zeta u$ has its support in $[0,T]\times(x_1',1)$, but it is also a solution of the uniformly parabolic equation
	\begin{equation}\label{equ:3.40}
	\begin{cases}
	Z_t-(a(x) Z_x)_x+\tilde{c}Z= \zeta f -\zeta_x a(x) u_x-(a(x)\zeta_x u)_x, & (t,x)\in Q,\\
	Z(t,1)=0 \text{ and }\begin{cases}
	\ds Z(t,0)=0,\,\,\text{ in case (WD)}  \\
	\ds  {(a(x) Z_{ x})}(t,0)=0,\,\,\text{ in case (SD) }
	\end{cases} & \text{ on }  (0,T),\\
	Z(0,x)=\zeta(x)u_0, & x\in (0,1).
	\end{cases}
	\end{equation}
	 Hence, by inequality \cite[Lemma 1.2]{Fursikov},   we have
	\begin{align*}
	\int\!\!\!\!\!\int_{Q}&\Big( \frac{1}{s\theta} \left( Z_t^2+( Z_{x x})^2\right)+s\theta a Z_x^2+s^3\theta^3\frac{x^2}{a} Z^2 \Big) e^{2s\varPhi}\, dxdt\\
	&\leqslant C \Big(\int\!\!\!\!\!\int_{Q} \left(\zeta^2 f^2+
	\big(\zeta_x  a(x)
	u_x+(a(x) \zeta_x u)_x\big)^2\right) e^{2s\varPhi}\, dxdt+\int\!\!\!\!\!\int_{\omega\times(0,T)}s^3\theta^3 Z^2  e^{2s\varPhi}\, dxdt\Big).
	\end{align*}
	Again, from the definition of $\zeta$ and the Cacciopoli inequality \cite[Lemma 6.1]{hjjaj}, we obtain
\small{	$$
	 \int\!\!\!\!\!\int_{Q} \big(\zeta_x  a(x)
	u_x+(a(x) \zeta_x u)_x\big)^2 e^{2s\varPhi}\, dxdt \leqslant  C \int\!\!\!\!\!\int_{Q_{\omega''}}(u^2+u_x^2) e^{2s\Phi} dtdx
		\leqslant C  \int\!\!\!\!\!\int_{Q_{\omega'}} \big(s^2\theta^2 u^2+f^2\big)e^{2s\Phi} dtdx.
	$$ }
	Thus
	\begin{align*}
	\int\!\!\!\!\!\int_{Q}\left( \frac{1}{s\theta} \left( Z_t^2+ Z_{x x}^2\right)\right.&+ \left. s\theta a Z_x^2+s^3\theta^3\frac{x^2}{a} Z^2 \right) e^{2s\varPhi}\, dxdt \\
	&\leqslant C \Big(\int\!\!\!\!\!\int_{Q} \zeta^2 f^2 e^{2s\varPhi}\, dxdt+\int\!\!\!\!\!\int_{\omega\times(0,T)}s^3\theta^3 u^2  e^{2s\varPhi}\, dxdt\Big).\numberthis\label{pro3.82}
	\end{align*} 	
From $\zeta u_x=Z_x-\zeta_x u$ and $\supp \zeta_x\Subset \omega''$, we deduce
\begin{align*}
 \ds \int\!\!\!\!\!\int_{Q} s\theta a \zeta^2 u_x^2  e^{2s\varPhi} dx dt& \leqslant  C\Big(\int\!\!\!\!\!\int_{Q} s\theta a Z_x^2  e^{2s\varPhi} dx dt+\int\!\!\!\!\!\int_{Q_{\omega''}} s\theta u^2  e^{2s\varPhi} dx dt\Big)	\\
 &\leqslant  C\Big(\int\!\!\!\!\!\int_{Q} s\theta a Z_x^2  e^{2s\varPhi} dx dt+\int\!\!\!\!\!\int_{Q_{\omega'}} s^3\theta^3 u^2  e^{2s\varPhi} dx dt\Big)\numberthis\label{pro3.83}
\end{align*}
for $s$ large enough. Similarly since $\ds \zeta u_{x x}=Z_{x x}-\zeta_{x x} u-2\zeta_x u_x$ and thanks to Cacciopoli inequality, we get
 \begin{align*}
 \ds \int\!\!\!\!\!\int_{Q}\frac{1}{ s\theta}  \zeta^2 u_{x x}^2  e^{2s\varPhi} dx dt& \leqslant  C\Big(\int\!\!\!\!\!\int_{Q} \frac{1}{ s\theta} Z_{x x}^2  e^{2s\varPhi} dx dt+ \int\!\!\!\!\!\int_{Q_{\omega''}} (\zeta_{x x} u+2\zeta_x u_x)^2  e^{2s\varPhi} dx dt\Big)	\\
 &\leqslant C\Big(\int\!\!\!\!\!\int_{Q} \frac{1}{ s\theta} Z_{x x}^2  e^{2s\varPhi} dx dt+\int\!\!\!\!\!\int_{Q_{\omega'}} \left(s^2\theta^2 u^2+f^2\right)e^{2s\Phi} dtdx\Big).\numberthis\label{pro3.84}
 \end{align*}
 The estimates \eqref{pro3.82}-\eqref{pro3.84} lead to
 \begin{align*}
 	 	\int\!\!\!\!\!\int_{Q}\left( \frac{1}{s\theta} \zeta^2u_t^2+\frac{1}{s\theta}\zeta^2( u_{x x})^2 \right.&+ \left.s\theta a \zeta^2 u_x^2+s^3\theta^3\frac{x^2}{a}\zeta^2 u^2 \right) e^{2s\varPhi}\, dxdt\\
 	 	&\leqslant C \Big(\int\!\!\!\!\!\int_{Q} \zeta^2 f^2 e^{2s\varPhi}\, dxdt+\int\!\!\!\!\!\int_{\omega\times(0,T)}s^3\theta^3 u^2  e^{2s\varPhi}\, dxdt\Big).\numberthis\label{pro3.85}
 \end{align*}	
Since $a$ is continuous on $(x_1',1]$ and by using  $\ds \mathcal{M}u= a'u_x+au_{x x}$ we obtain
\begin{align*}
	\int\!\!\!\!\!\int_{Q}\frac{1}{s\theta}\zeta^2 (\mathcal{M}u)^2 e^{2s\Phi} dtdx&\leqslant   C\Big(\int\!\!\!\!\!\int_{Q} \frac{1}{s\theta}\zeta^2u_{x x}^2\,dxdt+\int\!\!\!\!\!\int_{Q} s\theta a \zeta^2 u_x^2\,dxdt\Big).\numberthis\label{pro3.86}
\end{align*}	
	Thus, combining \eqref{pro3.85} and \eqref{pro3.86} we deduce the desired estimate.
\end{proof}
Again,  Proposition \ref{estimaa1} can be generalized as follows.
\begin{proposition}\label{proposition26}
		Let $T>0$ and $\tau\in\R$. Then, there exist two
		positive constants $C$ and $s_0$ such that for every $u_0\in L^2(0,1)$, the solution $u$ of equation  \eqref{problem} satisfies
		\begin{multline}\label{eqproposition26}
		\int\!\!\!\!\!\int_{Q} \Big(s^{\tau-1}\theta^{\tau-1}\zeta^2 u_t^2+s^{\tau-1}\theta^{\tau-1}\zeta^2 (\mathcal{M}u)^2+s^{1+\tau}\theta^{1+\tau} a\zeta^2 u_x^2+s^{3+\tau} {\theta}^{3+\tau}\frac{x^2}{a}\zeta^2 u^2 \Big)e^{2s\Phi}dx dt \\
		\leq C\Big(\int\!\!\!\!\!\int_{Q}  \zeta^2 s^\tau \theta^{\tau}f^2(t,x)e^{2s\Phi}dx dt+\int\!\!\!\!\!\int_{Q_{\omega'}} s^{3+\tau}\theta^{3+\tau} u^2 e^{2s\Phi}  dx dt\Big)
		\end{multline}
		for all $s\geqslant s_0$, 	with $\zeta=1-\xi$.
	\end{proposition}
	
Now we examine the case of a scalar degenerate parabolic equation ??.
\subsection{Carleman estimate for a scalar degenerate parabolic equation}\noindent\newline
In this section we will consider $z$, with the monomial derivative
  $\mathcal{M}^i\partial_t^j z \in L^2(0,T;H_a^2(0,1))$  for every $i,j\in\N$, a solution of the following scalar degenerate parabolic equation of order $2n$ in space.
\begin{equation}\label{scalar:parabolic}
\begin{cases}
P(\partial_t,\mathcal{M})z=0 & \text{in } Q,\\
\mathbf{C}\mathcal{M}^k z=0 & \text{on } \Sigma, \forall k\geqslant 0,
\end{cases}
\end{equation}
where $P(\partial_t,\mathcal{M})$ is the operator defined  by $\ds P(\partial_t,\mathcal{M})= det (\partial_t I_d+\mathbf{D}^* \mathcal{M}+A^*).$ Since the matrix $\mathbf{D}$ is diagonalizable \eqref{D:diagonalizable}, one gets
\begin{align*}
	P(\partial_t,\mathcal{M})&= det (\partial_t I_d+\mathbf{D}^* \mathcal{M}+A^*)\\
	                    &= det (\partial_t I_d+ P^* \mathbf{J}{P^*}^{-1} \mathcal{M}+A^*)\\
	                    &= det (\partial_t I_d+ \mathbf{J} \mathcal{M}+ {P^*}^{-1}A^* P^*)\\
	                    &=P_n\cdots P_1+\sum\limits_{p=2}^{n-1}\sum\limits_{1\leqslant i_1<\cdots<i_p\leqslant n} \alpha_{i_1,\cdots,i_p}P_{i_1}\cdots P_{i_p}+\sum\limits_{i=1}^{n}\alpha_i P_i +\alpha ,
\end{align*}
where $\ds P_i\equiv \partial_t+d_i \mathcal{M},\,\,1\leqslant i\leqslant  n,\,\,\, d_i>1$ and $\alpha_{i_1,\cdots,i_p},\,\alpha_i,\,\alpha\in\R$ depend only on the matrices $\mathbf{D}$ and $A$. The main result in this subsection is the following.
 \begin{theorem}\label{theorem:3.2}
	Let us fix $k_1$, $k_2\in \N$ and $\tau_0\in\R$. Then, there exist two positive constants $C_0$ and $s_0$ (only depending in  $\omega$, $n$, $a$, $\mathbf{D}$, $A$, $\tau_0$, $k_1$ and $k_2$) and $r=r(n)\in \N$ such the following inequality
	\begin{equation}
		\sum\limits_{i=0}^{k_1}\sum\limits_{j=0}^{k_2}\mathcal{J}(\tau_0-4(i+j),\mathcal{M}^i\partial_t^j\phi)\leqslant C_0 \iint_{Q_{\omega}}(s\theta)^{\tau_0+r}e^{2s\varPhi}{|\phi|}^2
	\end{equation}
holds for all $s\geqslant s_0$ and for every solution $\phi$ of equation \eqref{scalar:parabolic} that satisfies $\ds \mathcal{M}^i\partial_t^j\phi \in L^2(0,T,H^2_a(0,1))$ for every $i,j\in\N$. The terms $\ds \mathcal{J}(\tau,\phi)$ and $\ds I(\tau,z)$  are given by
\begin{equation}
\left\{
\begin{array}{lll}
\ds\mathcal{J}(\tau,\phi)=I(\tau+3(n-1),\phi)+\sum\limits_{i=2}^{n}I(\tau+3(n-2),P_{i}\phi)\\
\ds\qquad\qquad\qquad+\sum\limits_{p=2}^{n-1}\sum\limits_{1\leqslant i_1<\cdots<i_p\leqslant n}I(\tau+3(n-p-1),P_{i_p}\cdots P_{i_1}\phi),\\
\ds I(\tau,z)=\int\!\!\!\!\!\int_{Q}\Big(s^{\tau-1}\theta^{\tau-1} z_t^2 e^{2s\varphi}+s^{\tau-1}\theta^{\tau-1} (\mathcal{M}z)^2 e^{2s\varphi} \\ \ds\qquad\qquad\qquad+s^{\tau+1}\theta^{\tau+1}e^{2s\varphi}a(x)z_x^2
+	s^{\tau+3}\theta^{\tau+3}e^{2s\varphi}\frac{x^2}{ a(x)}z^2 \Big) dx dt.
\end{array}
\right.
\end{equation}
\end{theorem}
\begin{proof}
	Adapting the technique used by Ammar-Khodja et al. in \cite{A4} to our degenerate case, the proof will be divided in three steps. All along this proof $C$ will be a generic constants
	that may depend on $\omega$, $n$, $a$, $\mathbf{D}$, $A$, $\tau_0$, $k_1$ and $k_2$.\\
{\textbf{Step 1 :}}
	 Let us denote
	 \begin{equation}\label{f(z):def}
	 	F(z)=-\left(\sum\limits_{p=2}^{n-1}\sum\limits_{1\leqslant i_1<\cdots<i_p\leqslant n} \alpha_{i_1,\cdots,i_p}P_{i_1}\cdots P_{i_p}+\sum\limits_{i=1}^{n}\alpha_i P_i +\alpha\right)z,
	 \end{equation}
	 and consider the following change of variables
	 \begin{equation}\label{equ:3.5}
	 	\begin{cases}
	 	\psi_1=z,\\
	 	\psi_{i}=P_{i-1}\psi_{i-1}=(\partial_t+d_{i-1}\mathcal{M})\psi_{i-1},\quad 2\leqslant i\leqslant n.
	 	\end{cases}
	 \end{equation}
Having in mind the regularity assumptions on $z$, \eqref{scalar:parabolic} and \eqref{f(z):def}, one gets $\psi_i,\, F(z)\in L^2(Q)$ for every $i,\,\,1\leqslant i\leqslant n$ and $\Psi=(\psi_1,\cdots,\psi_n)^*$ satisfies the following cascade system
\begin{equation}\label{equ:3.6}
	\begin{cases}
	(\partial_t+d_1 \mathcal{M})\psi_{1}=\psi_2 & \text{ in } Q, \\
	(\partial_t+d_2 \mathcal{M})\psi_{2}=\psi_3 & \text{ in } Q, \\
	                                & \vdots         \\
	(\partial_t+d_n \mathcal{M})\psi_{n}=F(z) & \text{ in } Q, \\
	\mathbf{C}\psi_i=0 \text{ on } \Sigma   &1\leqslant i\leqslant n.
	\end{cases}
\end{equation}
For $i=1,\cdots,n-1$, applying respectively Proposition \ref{propo3.6} and Proposition \ref{proposition26} and combining the two estimates obtained leads to
 \begin{equation}
 		\ds I(\tau_0+3(n-i),\psi_i)\leqslant C \Big(\int\!\!\!\!\!\int_{Q}(s\theta)^{\tau_0+3(n-i)}e^{2s\Phi}{|\psi_{i+1}|}^2+ \int\!\!\!\!\!\int_{Q_{\omega'}}(s\theta)^{\tau_0+3(n-i+1)}e^{2s\Phi}{|\psi_{i}|}^2\Big).
 \end{equation}
 And  for $i=n$, we obtain
  \begin{equation}
  \ds I(\tau_0,\psi_n)\leqslant C \Big(\int\!\!\!\!\!\int_{Q}(s\theta)^{\tau_0}e^{2s\Phi}{|F(z)|}^2+ \int\!\!\!\!\!\int_{Q_{\omega'}}(s\theta)^{\tau_0+3}e^{2s\Phi}{|\psi_{n}|}^2\Big)
  \end{equation}
  for every $s\geqslant s_0$. Thus, a suitable combination of  the above inequalities leads to
\begin{equation}\label{equ:3.9}
	\sum\limits_{i=1}^{n}I(\tau_0+3(n-i),\psi_i)\leqslant C\Big(\sum\limits_{i=1}^{n}\int\!\!\!\!\!\int_{Q_{\omega'}}(s\theta)^{\tau_0+3(n-i+1)}e^{2s\Phi}{|\psi_{i}|}^2+\int\!\!\!\!\!\int_{Q}(s\theta)^{\tau_0}e^{2s\Phi}{|F(z)|}^2 \Big).
\end{equation}
 {\textbf{Step 2 :}}\\
  For $i=1,\cdots,n$, let us introduce the following sequence $(\mathcal{O}_i)_{1\leqslant i\leqslant n}$ of open sets and an associated family of truncation functions $(\chi_i)_{1\leqslant i\leqslant n}$ such that
 \begin{equation}
 	\mathcal{O}_n=\omega'\Subset\mathcal{O}_{n-1}\Subset\cdots\Subset\mathcal{O}_1\equiv\omega\subset(0,1),
 \end{equation}
 and \begin{equation}
 	\begin{cases}
 	\chi_i\in \mathcal{C}^2_c(\mathcal{O}_{i-1}),\\
 	0\leqslant\chi_i\leqslant 1 \text{ in } \mathcal{O}_{i-1},\\
 	\chi_i=1 \text{ in } \overline{\mathcal{O}_i}.
 	\end{cases}
 \end{equation}
 Let $l\geqslant 3$ and $k\in \{2,\cdots,n\}$, we multiply the equation $\ds \partial_t\psi_{k-1}+d_{k-1}\mathcal{M}\psi_{k-1}=\psi_{k} $, satisfied by $\psi_{k-1}$, by $\ds\delta\,\chi_{k}\psi_{k}$ with $\ds \delta=(s\theta)^{\tau_0+l}e^{2s\varPhi}$ and integrate  on $Q$, we obtain

 \begin{align*}
 	\int\!\!\!\!\!\int_{(0,T)\times\mathcal{O}_{k}}(s\theta)^{\tau_0+l}&e^{2s\varPhi}{|\psi_{k}|}^2 dx\,dt = \int\!\!\!\!\!\int_{Q}\delta \,\chi_k{|\psi_{k}|}^2 dx\,dt\\
 	& = \int\!\!\!\!\!\int_{Q}\delta \,\chi_k(\partial_t\psi_{k-1}+d_{k-1}\mathcal{M}\psi_{k-1})\psi_{k}dx\,dt\\
 	&=\underset{I_1}{\underbrace{\int\!\!\!\!\!\int_{Q}\delta\,\chi_k\partial_t (\psi_{k-1}) \psi_k dx\,dt} }+
 	\underset{I_2}{\underbrace{\int\!\!\!\!\!\int_{Q} d_{k-1}\delta\,\chi_k \psi_k\mathcal{M}\psi_{k-1} dx\,dt } }.\numberthis\label{equ:3.12}
 \end{align*}
 For every $\nu\leqslant \mu $  and  $(t,x)\in Q$, we have
\begin{equation}
\begin{cases}
	 {(s\theta)}^{\nu}\leqslant {(s\theta)}^{\mu},\\
	 |{\left({(s\theta)}^{\nu}e^{2s\varPhi}\right)}_x|\leqslant {(s\theta)}^{\nu+1}e^{2s\varPhi},\\
	 |{\left({(s\theta)}^{\nu}e^{2s\varPhi}\right)}_t|\leqslant {(s\theta)}^{\nu+2}e^{2s\varPhi}.
\end{cases}
\end{equation}
We have
\begin{align*}
	%I_1&=\int\!\!\!\!\!\int_{Q}u\chi_k\partial_t \psi_{k-1} \psi_k  \\
	  I_1 &= -\underset{I_1^{(1)}}{\underbrace{\int\!\!\!\!\!\int_{Q}\chi_k \psi_{k-1} \psi_k \partial_t \delta dx\,dt}} -\underset{I_1^{(2)}}{\underbrace{\int\!\!\!\!\!\int_{Q}\delta\,\chi_k \psi_{k-1}\partial_t\psi_k dx\,dt}}.
\end{align*}
 Since the function $x\mapsto\frac{x^2}{a(x)}$ is  bounded on $\ds \mathcal{O}_{k-1}$, we have
 \begin{align*}
I_1^{(1)}& \leqslant C\iint_{Q}{(s\theta)}^{\tau_0+l+2}  e^{2s\varPhi}\chi_k \psi_{k-1} \psi_k dx\,dt \\
&\leqslant  C\int\!\!\!\!\!\int_{Q}\Big((s\theta)^{\frac{\tau_0}{2}+\frac{3(n-k)}{2}+\frac32}\sqrt{\chi_k}e^{s\varPhi}\sqrt{\frac{2\varepsilon}{C}}\sqrt{\frac{x^2}{a}} \psi_k \Big)dx\,dt\\ & \qquad\qquad\times\Big(  (s\theta)^{\frac{\tau_0}{2}+l-\frac{3(n-k)}{2}+\frac12}\sqrt{\chi_k}e^{s\varPhi}\sqrt{\frac{C}{2\varepsilon}}\sqrt{\frac{x^2}{a}}\psi_{k-1}\Big)dx\,dt\\
&\leqslant \varepsilon\int\!\!\!\!\!\int_{Q}(s\theta)^{\tau_0+3(n-k)+3}\chi_k e^{2 s\varPhi}\frac{x^2}{a} { \psi_k }^2dx\,dt +  \frac{C}{4\varepsilon}\int\!\!\!\!\!\int_{Q}(s\theta)^{\tau_0+2l-3(n-k)+1}\chi_k e^{2 s\varPhi}\frac{x^2}{a}\psi_{k-1}^2dx\,dt.
\end{align*}
Likewise, we get
	\begin{align*}
	%I_1^{(2)}&=\int\!\!\!\!\!\int_{Q}u\chi_k \psi_{k-1}\partial_t\psi_k dxdt\\
	I_1^{(2)}&\leqslant  \int\!\!\!\!\!\int_{Q}{(s\theta)}^{\tau_0+l} e^{2s\varPhi} \chi_k \psi_{k-1} \partial_t\psi_k dx\,dt  \\
	&\leqslant \int\!\!\!\!\!\int_{Q}\Big({(s\theta)}^{\frac{\tau_0}{2}+\frac{3(n-k)}{2}-\frac12} e^{s\varPhi}\sqrt{\chi_k}\partial_t \psi_{k}\sqrt{2\varepsilon}\Big)\times\Big({(s\theta)}^{\frac{\tau_0}{2}+l-\frac{3(n-k)}{2}+\frac12} e^{s\varPhi}\sqrt{\chi_k} \psi_{k-1}\sqrt{\frac{1}{2\varepsilon}}\Big)dx\,dt \\
	&\leqslant \varepsilon \int\!\!\!\!\!\int_{Q}
	{(s\theta)}^{\tau_0 +3(n-k)-1} e^{2s\varPhi}\chi_k\partial_t \psi_{k}^2dx\,dt+\frac{1}{4\varepsilon} \int\!\!\!\!\!\int_{Q}{(s\theta)}^{\tau_0+2l-3(n-k)+1} e^{2s\varPhi}\chi_k \psi_{k-1}^2 dx\,dt.
	\end{align*}
Therefore
\begin{align*}
	I_1&\leqslant  \varepsilon I(\tau_0+3(n-k),\psi_k)+\frac{C}{2\varepsilon} \int\!\!\!\!\!\int_{(0,T)\times\mathcal{O}_{k-1}}{(s\theta)}^{\tau_0+2l-3(n-k)+1} e^{2s\varPhi}\psi_{k-1}^2 dx\,dt.
\end{align*}
On the other hand, for $I_2$ we have 	
\begin{align*}
	 \int\!\!\!\!\!\int_{Q}\mathcal{M}( \delta\,\chi_k \psi_k)  \psi_{k-1} dxdt & =\underset{I_2^{(1)}}{\underbrace{\int\!\!\!\!\!\int_{Q}  \psi_k  \psi_{k-1}\mathcal{M}( \delta\,\chi_k)dxdt}}+\underset{I_2^{(2)}}{\underbrace{\int\!\!\!\!\!\int_{Q}2a( \delta\,\chi_k)_x( \psi_k)_x  \psi_{k-1}dxdt}}\\
	 &\qquad\qquad+\underset{I_2^{(3)}}{\underbrace{\int\!\!\!\!\!\int_{Q} \delta\,\chi_k \psi_{k-1}\mathcal{M}(\psi_k) dxdt}}.
\end{align*}
Since $Supp(\mathcal{M}(\delta\,\chi_k))\subset\mathcal{O}_{k-1} $, $|\mathcal{M}(\delta\,\chi_k)|\leqslant C(s\theta)^{\tau_0+l+2}e^{2s\varPhi}$ and the function $\ds x\mapsto \frac{x^2}{a}$ is bounded on $\mathcal{O}_{k-1}$, then
\begin{align*}
&I_2^{(1)}\leqslant C \int\!\!\!\!\!\int_{(0,T)\times\mathcal{O}_{k}} (s\theta)^{\tau_0+l+2}e^{2s\varPhi} \psi_k  \psi_{k-1}  dxdt \\
&\leqslant C \int\!\!\!\!\!\int_{(0,T)\times\mathcal{O}_{k}} \Big((s\theta)^{\frac{\tau_0}{2}+\frac{3(n-k)}{2}+\frac32}e^{s\varPhi}\sqrt{\frac{\varepsilon}{C}}\sqrt{\frac{x^2}{a}} \psi_k \Big)\Big( (s\theta)^{\frac{\tau_0}{2}+l-\frac{3(n-k)}{2}+\frac12}e^{s\varPhi}\sqrt{\frac{C}{\varepsilon}}\sqrt{\frac{x^2}{a}} \psi_{k-1}\Big)  dxdt \\
&\leqslant \frac{\varepsilon}{2} \int\!\!\!\!\!\int_{(0,T)\times\mathcal{O}_{k}} (s\theta)^{\tau_0+3(n-k)+3}e^{2s\varPhi}\frac{x^2}{a} \psi_k^2 dxdt +\frac{C}{2\varepsilon} \int\!\!\!\!\!\int_{(0,T)\times\mathcal{O}_{k-1}} (s\theta)^{\tau_0+2l-3(n-k)+1}e^{2s\varPhi}\frac{x^2}{a} \psi_{k-1}^2 dxdt.\numberthis\label{th:I_2^{(1)}}
\end{align*}
Likewise, as $Supp(\delta\,\chi_k)\subset\mathcal{O}_{k-1} $ and $\Big|(\delta\,\chi_k)_x\Big|\leqslant C (s\theta)^{\tau_0+l+1}e^{2s\varPhi}$ we have
\begin{align*}
        &I_2^{(2)} \leqslant C\int\!\!\!\!\!\int_{(0,T)\times\mathcal{O}_{k}}(s\theta)^{\tau_0+l+1}e^{2s\varPhi} ( \psi_k)_x  \psi_{k-1} dxdt \\
        & \leqslant C\int\!\!\!\!\!\int_{(0,T)\times\mathcal{O}_{k}}\Big((s\theta)^{\frac{\tau_0}{2}+\frac{3(n-k)}{2}+\frac12}\sqrt{\frac{\varepsilon}{C}}\sqrt{a}e^{s\varPhi} ( \psi_k)_x  \Big)\Big((s\theta)^{\frac{\tau_0}{2}+l-\frac{3(n-k)}{2}+\frac12}\sqrt{\frac{C}{\varepsilon}}\sqrt{\frac{x^2}{a}}e^{s\varPhi}  \psi_{k-1}\Big) dxdt \\
        & \leqslant \frac{\varepsilon}{2} \int\!\!\!\!\!\int_{(0,T)\times\mathcal{O}_{k}}(s\theta)^{\tau_0+3(n-k)+1}a e^{2s\varPhi} ( \psi_k)_x^2  dxdt+ \frac{C}{2\epsilon} \int\!\!\!\!\!\int_{(0,T)\times\mathcal{O}_{k}} (s\theta)^{\tau_0+2l-3(n-k)+1}\frac{x^2}{a}e^{2s\varPhi}  \psi_{k-1}^2 dxdt \\
        & \leqslant \frac{\varepsilon}{2} \int\!\!\!\!\!\int_{(0,T)\times\mathcal{O}_{k}}(s\theta)^{\tau_0+3(n-k)+1}a e^{2s\varPhi} ( \psi_k)_x^2  dxdt+ \frac{C}{2\epsilon} \int\!\!\!\!\!\int_{(0,T)\times\mathcal{O}_{k-1}} (s\theta)^{\tau_0+2l-3(n-k)+1}\frac{x^2}{a}e^{2s\varPhi}  \psi_{k-1}^2 dxdt.\numberthis\label{th:I_2^{(2)}}
\end{align*}

\begin{align*}
	    &I_2^{(3)}\leqslant C \int\!\!\!\!\!\int_{Q}  (s\theta)^{\tau_0+l}e^{2s\varPhi}\chi_k \mathcal{M}(\psi_k)  \psi_{k-1} dxdt \\
	    &\leqslant C \int\!\!\!\!\!\int_{Q} \Big( (s\theta)^{\frac{\tau_0}{2}+\frac{3(n-k)}{2}-\frac12}\sqrt{\chi_k}e^{s\varPhi}\sqrt{\frac{\varepsilon}{C}} (\mathcal{M}\psi_k)\Big) \Big( (s\theta)^{\frac{\tau_0}{2}+l-\frac{3(n-k)}{2}+\frac12}\sqrt{\chi_k}\sqrt{\frac{x^2}{a}}e^{s\varPhi}\sqrt{\frac{C}{\varepsilon}} \psi_{k-1}\Big) dxdt \\
	    &\leqslant \frac{\varepsilon}{2} \int\!\!\!\!\!\int_{Q} (s\theta)^{\tau_0+3(n-k)-1}\chi_k e^{2s\varPhi} (\mathcal{M}\psi_k)^2 dxdt+ \frac{C}{2\varepsilon} \int\!\!\!\!\!\int_{Q} (s\theta)^{\tau_0+2l-3(n-k)+1} \chi_k \frac{x^2}{a} e^{2s\varPhi} \psi_{k-1}^2 dxdt \\
& \leqslant \frac{\varepsilon}{2} \int\!\!\!\!\!\int_{Q} (s\theta)^{\tau_0+3(n-k)-1}\chi_k e^{2s\varPhi} (\mathcal{M}\psi_k)^2 dxdt+ \frac{C}{2\varepsilon} \int\!\!\!\!\!\int_{(0,T)\times\mathcal{O}_{k-1}} (s\theta)^{\tau_0+2l-3(n-k)+1}\frac{x^2}{a} e^{2s\varPhi} \psi_{k-1}^2 dxdt. \numberthis\label{th:I_2^{(3)}}
\end{align*}
From \eqref{th:I_2^{(1)}}-\eqref{th:I_2^{(3)}} we get
\begin{equation}
	I_2\leqslant \varepsilon I(\tau_0+3(n-k),\psi_{k})+ \frac{C}{\varepsilon} \int\!\!\!\!\!\int_{(0,T)\times\mathcal{O}_{k-1}} (s\theta)^{\tau_0+2l-3(n-k)+1}\frac{x^2}{a} e^{2s\varPhi} \psi_{k-1}^2 dxdt.
\end{equation}
 Coming back to \eqref{equ:3.12} we get
 \begin{align*}
 \int\!\!\!\!\!\int_{(0,T)\times\mathcal{O}_{k}}(s\theta)^{\tau_0+l}e^{2s\varPhi}{|\psi_{k}|}^2
 &\leqslant \varepsilon I(\tau_0+3(n-k),\psi_{k})\\
 &\qquad\qquad+\frac{C}{\varepsilon} \int\!\!\!\!\!\int_{(0,T)\times\mathcal{O}_{k-1}} (s\theta)^{\tau_0+2l-3(n-k)+1}\frac{x^2}{a} e^{2s\varPhi} \psi_{k-1}^2 dxdt,\numberthis\label{equ:3.14}
 \end{align*}
 with $\varepsilon>0$.
 	For $l=2$ and $\varepsilon=\frac{1}{2C}$, where $C$ is the constant used in \eqref{equ:3.9}
 	 \begin{align*}
 	 \int\!\!\!\!\!\int_{(0,T)\times\mathcal{O}_{n}}(s\theta)^{\tau_0+2}e^{2s\varPhi}{|\psi_{n}|}^2
 	 &\leqslant \frac{1}{2C} I(\tau_0,\psi_{n})+2C^2 \int\!\!\!\!\!\int_{(0,T)\times\mathcal{O}_{n-1}} (s\theta)^{\tau_0+5}\frac{x^2}{a} e^{2s\varPhi} \psi_{n-1}^2 dxdt.\numberthis\label{cc3.49}
 	 \end{align*}

 So, from \eqref{equ:3.9} and \eqref{cc3.49} we infer
\begin{equation*}\label{}
\sum\limits_{i=1}^{n}I(\tau_0+3(n-i),\psi_i)\leqslant C\Big(\sum\limits_{i=1}^{n-1}\int\!\!\!\!\!\int_{(0,T)\times\mathcal{O}_{n-1}}(s\theta)^{\tau_0+k(i)}e^{2s\Phi}{|\psi_{i}|}^2dxdt+\int\!\!\!\!\!\int_{Q}(s\theta)^{\tau_0}e^{2s\Phi}{|F(z)|}^2dxdt \Big),
\end{equation*}
where $k(i)=\max(5,3(n-i+1))$.\\
By iterating this operation $(n-1)$ times, there exist a positive constant $C>0$ and an integer $K=K(n)$ such that
\begin{equation}\label{}
\sum\limits_{i=1}^{n}I(\tau_0+3(n-i),\psi_i)\leqslant C\Big(\int\!\!\!\!\!\int_{(0,T)\times\omega}(s\theta)^{\tau_0+K}e^{2s\Phi}
{|\psi_1|}^2dxdt+\int\!\!\!\!\!\int_{Q}(s\theta)^{\tau_0}e^{2s\Phi}{|F(z)|}^2dxdt \Big),
\end{equation}
which in view of \eqref{equ:3.5} implies
\begin{align*}
& I(\tau_0+3(n-1),z)+\sum\limits_{i=2}^{n}I(\tau_0+3(n-i),P_{i-1}\cdots P_{1}z)\\
&\qquad \leqslant C\Big( \int\!\!\!\!\!\int_{(0,T)\times\omega} (s\theta)^{\tau_0+K}e^{2s\Phi}{|z|}^2dxdt +\int\!\!\!\!\!\int_{Q}(s\theta)^{\tau_0}e^{2s\Phi}{|F(z)|}^2 dxdt \Big).\numberthis\label{equ:3.15}	
\end{align*}
Now, at this level the left-hand-side of \eqref{equ:3.15} does not contains enough terms to absorb the
term corresponding to $F(z)$. So, in order to absorb the term $F(z)$, let $\varPi$ denote then any permutation of the set $\{1,2,\cdots,n\}$ and consider, instead of \eqref{equ:3.5}, the new change of variable
\begin{equation}
\begin{cases}
\psi_1=z,\\
\psi_{i}=P_{\varPi(i-1)}\psi_{i-1}=(\partial_t+d_{\varPi(i-1)}\mathcal{M})\psi_{i-1},\quad 2\leqslant i\leqslant n.
\end{cases}
\end{equation}
Then system \eqref{equ:3.6} becomes
\begin{equation*}
\begin{cases}
(\partial_t+d_{\varPi(1)} \mathcal{M})\psi_{1}=\psi_2 & \text{ in } Q, \\
(\partial_t+d_{\varPi(2)} \mathcal{M})\psi_{2}=\psi_3 & \text{ in } Q, \\
& \vdots         \\
(\partial_t+d_{\varPi(n)} \mathcal{M})\psi_{n}=F(z) & \text{ in } Q, \\
\mathbf{C}\psi_i=0 \text{ on } \Sigma   & \forall i~:~ 1\leqslant i\leqslant n.
\end{cases}
\end{equation*}
The same procedure as above leads to a similar estimate as \eqref{equ:3.15} which reads then
\begin{align*}
&I(\tau_0+3(n-1),z)+\sum\limits_{i=2}^{n}I(\tau_0+3(n-i),P_{\varPi(i-1)}\cdots P_{\varPi(1)}z)\\
&\leqslant C\Big(\int\!\!\!\!\!\int_{(0,T)\times\omega}(s\theta)^{\tau_0+K}e^{2s\Phi}{|z|}^2 dxdt +\int\!\!\!\!\!\int_{Q}(s\theta)^{\tau_0}e^{2s\Phi}{|F(z)|}^2 dxdt \Big).\numberthis\label{equ:3.17}	
\end{align*}
Now, considering all such possible permutations with associated change of variable, we finally obtain
\begin{align*}
& I(\tau_0+3(n-1),z)+\sum\limits_{i=2}^{n}I(\tau_0+3(n-i),P_{i}z)\\
&\quad+\sum\limits_{p=2}^{n-1}\sum\limits_{1\leqslant i_1<\cdots<i_p\leqslant n}I(\tau_0+3(n-p-1),P_{i_p}\cdots P_{i_1}z)\\
&\qquad \leqslant C\Big(\int\!\!\!\!\!\int_{(0,T)\times\omega}(s\theta)^{\tau_0+K}e^{2s\Phi}{|z|}^2 dxdt +\int\!\!\!\!\!\int_{Q}(s\theta)^{\tau_0}e^{2s\Phi}{|F(z)|}^2 dxdt \Big).\numberthis\label{equ:3.18}	
\end{align*}
From the definition of $F(z)$ \eqref{f(z):def}, we deduce
\begin{equation}\label{equ:3.19}
\int\!\!\!\!\!\int_{Q}(s\theta)^{\tau_0}e^{2s\Phi}{|F(z)|}^2 dxdt\leqslant C \int\!\!\!\!\!\int_{Q}(s\theta)^{\tau_0}e^{2s\Phi}\left(|z|^2+\sum\limits_{i=1}^{n}|P_i z|^2+\sum\limits_{p=2}^{n-1}\sum\limits_{1\leqslant i_1<\cdots<i_p\leqslant n}|P_{i_p}\cdots P_{i_1}z|^2\right)dxdt.
\end{equation}
Choosing $s$  large enough such that
$C(s\theta)^{\tau_0}\leqslant\frac{1}{2}(s\theta)^{\tau_0+2(n-p)},\,\,\forall p: 0\leqslant p\leqslant n-1, $ and
from  \eqref{equ:3.18} and \eqref{equ:3.19} we get

\begin{align*}
& I(\tau_0+3(n-1),z)+\sum\limits_{i=2}^{n}I(\tau_0+3(n-i),P_{i}z)\\
&+\sum\limits_{p=2}^{n-1}\sum\limits_{1\leqslant i_1<\cdots<i_p\leqslant n}I(\tau_0+2(n-p-1),P_{i_p}\cdots P_{i_1}z)\leqslant C\int\!\!\!\!\!\int_{(0,T)\times\omega}(s\theta)^{\tau_0+K}e^{2s\Phi}{|z|}^2.\numberthis\label{equ:3.20}	
\end{align*}
This can be written as
\begin{equation*}
	\ds \mathcal{J}(\tau_0,z)\leqslant C\int\!\!\!\!\!\int_{(0,T)\times\omega}(s\theta)^{\tau_0+K}e^{2s\Phi}{|z|}^2.
\end{equation*}
{\textbf{Step  3:}}\\
  From the regularity assumptions imposed on $z$, if $1\leqslant i\leqslant k_1$ and $1\leqslant j\leqslant k_2$,  $\mathcal{M}^i\partial_t^j z $ also satisfies the equation \eqref{scalar:parabolic}. Therefore, by applying the two preceding steps to $\mathcal{M}^i\partial_t^j z $, there exist two two positive constants $C_{\tau}$ and $s_{\tau}$ such that:
 	\begin{equation}
 	\mathcal{J}(\tau,\mathcal{M}^i\partial_t^j z)\leqslant C_{\tau}\int\!\!\!\!\!\int_{(0,T)\times\omega} (s\theta)^{\tau+K}e^{2s\Phi}{|\mathcal{M}^i\partial_t^j z|}^2\;dx\,dt,
 	\end{equation}
 for every $s\geqslant s_{\tau}$
 	
 Now for $s$ large enough we have
 \begin{align*}
 	 \int\!\!\!\!\!\int_{(0,T)\times\omega} (s\theta)^{\tau}e^{2s\Phi}\left({|\mathcal{M} z|}^2+{|\partial_t z|}^2 \right)dxdt %& \leqslant  \int\!\!\!\!\!\int_{(0,T)\times\omega} (s\theta)^{\tau+3}e^{2s\Phi}\left({|\mathcal{M} z|}^2+{|\partial_t z|}^2 \right)dxdt\\
 	 & \leqslant I(\tau+4,z)\\
 	 %& \leqslant I(\tau+4+3(n-1),z)\\
 	 & \leqslant \mathcal{J}(\tau+4,z).
 \end{align*}
 Thus by iterating this process, one gets
 \begin{equation}
 \begin{cases}
 \mathcal{J}(\tau,\mathcal{M}^i\partial_t^{j}z)\leqslant C \mathcal{J}(\tau +4+K,\mathcal{M}^{i-1}\partial_t^{j}z),\\
 \mathcal{J}(\tau,\mathcal{M}^i\partial_t^{j}z)\leqslant C \mathcal{J}(\tau +4+K,\mathcal{M}^i\partial_t^{j-1}z).
 \end{cases}
 \end{equation}
 Therefore, by applying successively the last inequalities we deduce
 \begin{equation}
 \mathcal{J}(\tau,\mathcal{M}^i\partial_t^{j}z)\leqslant C \mathcal{J}(\tau +(4+K)(i+j),z).
 \end{equation}

Thus, taking in account \eqref{equ:3.20}, we get
\begin{align*}
\sum\limits_{i=0}^{k_1}\sum\limits_{j=0}^{k_2}\mathcal{J}(\tau_0-4(i+j),\mathcal{M}^i\partial_t^j z)
&\leqslant  \sum\limits_{i=0}^{k_1}\sum\limits_{j=0}^{k_2}C\mathcal{J}(\tau_0+K(i+j),z)\\
&\leqslant  C\mathcal{J}(\tau_0+K(k_1+k_2),z).\end{align*}
Consequently, we have
\begin{align*}
\sum\limits_{i=0}^{k_1}\sum\limits_{j=0}^{k_2}\mathcal{J}(\tau_0-4(i+j),\mathcal{M}^i\partial_t^j z)
&\leqslant C_0 \iint_{\omega_T}(s\theta)^{\tau_0+r}e^{2s\varPhi}{|\phi|}^2dx\,dt,
\end{align*}
where $r=K(k_1+k_2+1)$.
\end{proof}

% % % % % % % % % % % % % % % % % % % % % % % % % % % % % % % % % % % % % % % % % % %
% % % % % % % % % % % % % % % % % % % % % % % % % % % % % % % % % % % % % % % % % % %
\section{Null controllability of problem \eqref{syst1}}

Now, we will show the Carleman estimate for the adjoint problem \eqref{syst2adjoint}. Recall \\ $\ds\mathbb{D}=\cap_{p=0}^{\infty}D(\mathcal{M}^p)$ which is dense in $L^2(0,1)$. We have the following result see \cite[Proposition 3.3.]{A4}
\begin{proposition}\label{pro:3.3}
	Let $\varphi_0\in \mathbb{D}^n$ and let $\varphi=(\varphi_1,\cdots,\varphi_n)^*$ be the corresponding solution of problem \eqref{syst2adjoint}. Then, $\varphi \in \mathcal{C}^k([0,T];D(\mathcal{M}^p)^n)$ for every $k,\,p\geqslant 0$, and for every $i$ (with $1\leqslant i\leqslant n$) $\varphi_i$ solves equation \eqref{scalar:parabolic}. 
\end{proposition}

\begin{theorem}\label{thm4.2}
	Assume $\mathbf{D}$ satisfies the condition \eqref{D:diagonalizable},  Then, given  $\tau\in\R$ and
$k\geqslant (n-1)(2n-1)$, there exist $r=r(n)\in\N$ and two positive constants $C$ and  $\sigma$  such that for every $\varphi_0\in L^2(0,1)^n$ the corresponding solution $\varphi$ to the adjoint problem \eqref{syst2adjoint}  satisfies
\begin{equation}\label{equ:thm4.2}
	\int_{0}^{T}(s\theta)^{\tau}e^{-2sM_0\theta}\| \mathcal{M}^k\mathcal{K}^*\varphi\|_{L^2(0,1)^n}^2\leqslant C\int\!\!\!\!\!\int_{\omega_T}(s\theta)^{\tau+\kappa+r}e^{2s\Phi}|B^*\varphi|^2 ,
\end{equation}
where $\ds M_0=\underset{x\in (0,1)}{\max}\psi(x) $ and $\kappa=4k+n-4$.
\end{theorem}
\begin{proof}
	 Assume $\varphi_0\in \mathbb{D}^n$ and let $\varphi$ the solution of the adjoint problem \eqref{syst2adjoint} corresponding to $\varphi_0$. By Proposition \ref{pro:3.3} we have $\varphi\in\mathcal{C}^l([0,T];D(\mathcal{M}^p)^n)$ for every $l,\,p\geqslant 0$; and solves equation \eqref{scalar:parabolic}.
	 Likewise $(B^*\varphi)_j$ is in $\mathcal{C}^l([0,T];D(\mathcal{M}^p)^n)$; and solves equation \eqref{scalar:parabolic}, for all $j$ (with $1\leqslant j\leqslant m$).
From the expression of $\mathcal{K}^*$ we have
\begin{equation}
\mathcal{K}^*\varphi(t,\cdot)=\left((-1)^{n-1}\partial_t^{n-1} B^*\varphi,(-1)^{n-2}\partial_t^{n-2} B^*\varphi,\cdots,-\partial_t B^*\varphi, B^*\varphi\right)^*(t,\cdot)
\end{equation}
We have
\begin{align*}
\int_{0}^{T}(s\theta)^{\tau}e^{-2sM_0\theta}\| \mathcal{M}^k\mathcal{K}^*\varphi\|_{L^2(0,1)^n}^2
&\leqslant \sum_{i=1}^{m}\int\!\!\!\!\!\int_{Q}\sum_{j=0}^{n-1}(s\theta)^{\tau}e^{2s\Phi}|\mathcal{M}^k\partial_t^j (B^*\varphi)_i|^2 dxdt\\
&\leqslant \sum_{i=1}^{m}\sum_{j=0}^{n-1} I(\tau-3,\mathcal{M}^k\partial_t^j(B^*\varphi)_i) \\
&\leqslant \sum_{i=1}^{m}\sum_{j=0}^{n-1} \mathcal{J}(\tau-3-3(n-1),\mathcal{M}^k\partial_t^j(B^*\varphi)_i) \\
&\leqslant \sum_{i=1}^{m}\sum_{j=0}^{n-1} \mathcal{J}(\tau-3n,\mathcal{M}^k\partial_t^j(B^*\varphi)_i).
\end{align*}
For the choice of $\tau_0=\tau+4k+n-4$ one gets $\ds (s\theta)^{\tau}\leqslant C (s\theta)^{\tau_0-4(l+j)+3n}$ for every $l,\,j$ with $0\leqslant l\leqslant k$ and $0\leqslant j\leqslant n-1$. Thus, using Theorem \ref{theorem:3.2}, we deduce
\begin{align*}
\int_{0}^{T}(s\theta)^{\tau}e^{-2sM_0\theta}\| \mathcal{M}^k\mathcal{K}^*\varphi\|_{L^2(0,1)^n}^2
&\leqslant  \sum_{i=1}^{m}\sum_{j=0}^{n-1} \mathcal{J}(\tau-3n,\mathcal{M}^k\partial_t^j(B^*\varphi)_i) \\
&\leqslant \sum_{i=1}^{m}\sum_{j=0}^{n-1} \mathcal{J}(\tau_0-4(k+j),\mathcal{M}^k\partial_t^j(B^*\varphi)_i) \\
&\leqslant C \sum_{i=1}^{m} \int\!\!\!\!\!\int_{\omega_T}(s\theta)^{\tau_0+r}e^{2s\Phi}|(B^*\varphi)_i|^2\\
&\leqslant C\int\!\!\!\!\!\int_{\omega_T}(s\theta)^{\tau+\kappa+r}e^{2s\Phi}|B^*\varphi|^2,
\end{align*}
with $\kappa=4k+n-4$.\\
Now, when $\varphi_0\in L^2(0,1)^n$, there exists a Cauchy sequence $(\varphi_0^l)_{l\geqslant 1}\subset \mathbb{D}^n$ such that $\varphi_0^l\longrightarrow \varphi_0 \in L^2(0,1)^n$. Let $\varphi^l$ and $\varphi$ be, respectively, the solution of the adjoint problem \eqref{syst2adjoint} corresponding to $\varphi_0^l$ and $\varphi_0$, we have  $\varphi^l\longrightarrow \varphi$ in $L^2(0,1)^n$ and
$\mathcal{M}^k\mathcal{K}^*\varphi^l\longrightarrow \mathcal{M}^k\mathcal{K}^*\varphi$ in $\mathcal{D}'(Q)^n$ for every $k\geqslant 0$.  Since $\varphi^l$ satisfies \eqref{equ:thm4.2}, then we deduce that $(\mathcal{M}^k\mathcal{K}^*\varphi^l)_{l\geqslant 1}$ is a Cauchy sequence in the weighted space $L^2((s\theta)^{\frac{\tau}{2}}e^{-sM_0\theta},Q)$. Passing to the limit in the Carleman inequality \eqref{equ:thm4.2} satisfied by $\varphi^l$, we obtain the result in the general case. This ends the proof.
\end{proof}
At present, using the condition $Ker (\mathcal{K}^*)=\{0\}$ we state the following global Carleman estimate for the solution of Problem \eqref{syst2adjoint}

\begin{corollary}\label{coro4.3}
	 In addition to the assumptions in Theorem \ref{thm4.2}, we assume the condition $Ker (\mathcal{K}^*)=\{0\}$. Then, given $\tau\in\R$ and $k\geqslant (n-1)(2n-1)$, there exist two positive constants $C$ and $\sigma$ such that for every $\varphi_0\in L^2(0,1)^n$ the corresponding solution $\varphi$ to the adjoint problem \eqref{syst2adjoint} satifies
	 \begin{equation}
	 	\int\!\!\!\!\!\int_{Q} (s\theta)^{\tau}e^{-2sM_0\theta}{|\mathcal{M}^{k-(n-1)(2n-1)}\varphi|}^2\leqslant C \int\!\!\!\!\!\int_{(0,T)\times\omega}(s\theta)^{\tau+\kappa+r}e^{2s\Phi}|B^*\varphi|^2
	 \end{equation}
	 for every $s\geqslant \sigma$. $M_0,\,\kappa$ and $r=r(n)$ are as in Theorem \ref{thm4.2}.
\end{corollary}
\begin{proof}
	Since $Ker (\mathcal{K}^*)=\{0\}$ and $k\geqslant (2n-1)(n-1)$, then, we infer from Theorem \ref{theorem2.5}
	
	 $${\|\mathcal{M}^{k-(2n-1)(n-1)}\varphi \|}^2_{L^2(0,1)^n}\leqslant C  {\|\mathcal{M}^k \mathcal{K}^*\varphi \|}^2_{L^2(0,1)^{n m}}.$$
Now, using the inequality \eqref{equ:thm4.2} we have
\begin{align*}
	 	 	\int\!\!\!\!\!\int_{Q} (s\theta)^{\tau}e^{-2sM_0\theta}{|\mathcal{M}^{k_(n-1)(2n-1)}\varphi|}^2&\leqslant C \int\!\!\!\!\!\int_{Q} (s\theta)^{\tau}e^{-2sM_0\theta}{|\mathcal{M}^k \mathcal{K}^*\varphi|}^2\\ &\leqslant C\int\!\!\!\!\!\int_{(0,T)\times\omega}(s\theta)^{\tau+\kappa+r}e^{2s\Phi}|B^*\varphi|^2.
\end{align*}	
\end{proof}
At present, we are ready to give the proof of the main result.
\begin{proof}[Proof of Theorem \ref{maintheorem}]
The necessary part:
Suppose $Ker (\mathcal{K}^*)\neq\{0\}$, from Proposition \ref{prop22}, there exists $p_0\in\N^*$ such that $rank~\mathcal{K}_{p_0}=rank~[-\lambda_{p_0} \mathbf{D}+A|B]<n$. From the Kalman's rank condition applied to ordinary differential system
$$y'=(−\lambda_{p_0} \mathbf{D}+A)y + Bv$$ 	
is not controllable. Thus, there exists a nonzero solution $z_{p_0}(t)\in\R^n$ to the associated adjoint system
  	$$-z'=(−\lambda_{p_0} \mathbf{D}^*+A^*)z\,\,\, \text{ in } (0,T), $$
satisfying $B^*z_{p_0}(t)=0$ for all $t\in[0,T]$. Then,  let $\varphi_0=z_{p_0}(t)\varPhi_{p_0}$ where $\varPhi_{p_0}$ is the normalized eigenfunction associated with $\lambda_{p_0}$. The function $\varphi(t,x)=z_{p_0}(t)\varPhi_{p_0}$  is the solution of the adjoint problem \eqref{syst2adjoint},  corresponding to $\varphi_0$,  which  is nonzero and satisfies $B^*\varphi(t,x)=0$ in $Q$. So this solution does not satisfy the observability inequality \eqref{obser:inequality} and thus \eqref{syst1} is not controllable.	

For the sufficient part, let $\varphi\in L^2(0,T,H_a^1(0,1)^n)$ be the solution of Problem \eqref{syst2adjoint}  corresponding to $\varphi_0$. Using Corollary \ref{coro4.3} with $\tau=0$ and $k=(n-1)(2n-1)$, there exist two positive constants $C$ and $\sigma$ such that
\begin{equation}
	\int_{\frac T4}^{\frac{3T}4}\int_{0}^{1}e^{-2sM_0\theta}|\varphi|^2 dxdt\leqslant C\int\!\!\!\!\!\int_{(0,T)\times\omega}(s\theta)^l e^{2s\varPhi}|B^*\varphi|^2 dx dt,
\end{equation}
for all $s\geqslant \sigma$, where $l=4(n-1)(2n-1)+n+r$.
For all $t\in [\frac T4,\frac{3T}4]$,  we have $-2sM_0(\frac4{3T})^8\leqslant-2sM_0\theta\leqslant -2sM_0(\frac4{T})^8$, then we infer
$$e^{-2sM_0\theta}\geqslant e^{-2s M_0(\frac4{3T})^8},\,\, \forall t\in [\frac{T}{4},\frac{3T}{4}].$$ On the other hand, let $m_0=\min\limits_{x\in\omega}|\varPsi(x)|$. We have
\begin{align*}
	\int\!\!\!\!\!\int_{(0,T)\times\omega}(s\theta)^l e^{2s\varPhi}|B^*\varphi|^2 dx dt& \leqslant \int\!\!\!\!\!\int_{(0,T)\times\omega}(s\theta)^l e^{-2sm_0\theta}|B^*\varphi|^2 dx dt\\
	& \leqslant \int_{0}^T(s\theta)^l e^{-2sm_0\theta} \int_{\omega}|B^*\varphi|^2 dx dt.
\end{align*}
Since $\lim\limits_{t\rightarrow 0^+}(s\theta)^l e^{-2sm_0\theta}=\lim\limits_{t\rightarrow T^-}(s\theta)^l e^{-2sm_0\theta}=0$, we readily deduce
\begin{equation}\label{dem4.7}
\int_{\frac T4}^{\frac{3T}4}\int_{0}^{1}|\varphi|^2 dxdt\leqslant C\int\!\!\!\!\!\int_{(0,T)\times\omega}|B^*\varphi|^2 dx dt.
\end{equation}
As in \cite{A4}, there exists a positive constant $C$  depending on $\mathbf{D}$ and $A$ such that
\begin{equation}
	\frac{d}{dt}\left(e^{Ct}\|\varphi(t,\cdot)\|^2\right)\geqslant 0,\,\,\ \forall t\in (0,T).
\end{equation}
From this last inequality we also infer
\begin{align*}
	\|\varphi(0,\cdot)\|^2 \leqslant e^{C \frac{T}{4}}\|\varphi(\frac{T}{4},\cdot)\|^2
	\leqslant \frac{2}{T} e^{C \frac{3T}{4}}\int_{\frac{T}{4}}^{\frac{3T}{4}}\int_{0}^{1}|\varphi|^2.
\end{align*}
 Therefore, this last inequality together with \eqref{dem4.7}  imply
 the observability inequality for the solutions of the adjoint problem \eqref{syst2adjoint}
 \begin{equation}
 	 	\|\varphi(0,\cdot)\|^2 \leqslant  C \int\!\!\!\!\!\int_{(0,T)\times\omega}|B^*\varphi|^2 dx dt.
 \end{equation}	
 This completes the proof of the sufficient part and consequently that of Theorem \ref{maintheorem}.
 \end{proof}

% % % % % % % % % % % % % % % % % % % % % % % % % % % % % % % % % % % % % % % % % % % % % %
% %
% %                                References
% %
% % % % % % % % % % % % % % % % % % % % % % % % % % % % % % % % % % % % % % % % % % % % % %

%---------------------------------------------------------------
% REFERENCES BIBLIOGRAPHIQUES
%---------------------------------------------------------------
% NE PAS MODIFIER LES 2 LIGNES SUIVANTES
\bibliographystyle{plain}

\end{document}